\documentclass[11pt,letterpaper]{amsart}

\usepackage[margin=1.1in]{geometry}
\usepackage{url}
\usepackage[colorlinks=true, citecolor=blue]{hyperref}
\usepackage{bookmark}
    \usepackage{amssymb}
\usepackage{amsmath}
\usepackage{amscd,latexsym}
\usepackage{bookmark}
     \usepackage{amssymb,amsfonts,bm}
\usepackage[all,arc]{xy}
\usepackage{enumerate}
\usepackage{mathrsfs}
\usepackage{amscd}
\usepackage{tikz}
\usepackage{tikz-cd}
\usepackage{amsmath}
\usepackage{latexsym}
\usepackage{mathdots}
\usepackage{amsrefs}
\usepackage{graphicx}
\usepackage{mathtools}
\usepackage{leftidx}
\usepackage{tensor}
\usepackage{bbm}
\usepackage{dsfont}
\usepackage{enumitem}
\usepackage{xypic,amsthm, color,tabmacD}
\usepackage[new]{old-arrows}
\usepackage{enumitem,xcolor}
\usepackage{enumerate,mdwlist}
\usepackage{multicol}

\newcommand{\Tableau}[2][sY]{{\text{\tableau[#1]{#2}}}}

\makeatletter
\@namedef{subjclassname@2020}{%
  \textup{2020} Mathematics Subject Classification}
\makeatother

\theoremstyle{plane}
\newtheorem{theorem}{Theorem}[section]
\newtheorem{lemma}[theorem]{Lemma}
 \newtheorem{corollary}[theorem]{Corollary}
      \newtheorem{proposition}[theorem]{Proposition}

        \newtheorem{theorem*}{Theorem}

    \newcounter{relctr} 
\everydisplay\expandafter{\the\everydisplay\setcounter{relctr}{0}} 

\newcommand\labelrel[2]{%
  \begingroup
    \refstepcounter{relctr}%
    \stackrel{\textnormal{(\alph{relctr})}}{\mathstrut{#1}}%
    \originallabel{#2}%
  \endgroup
}
\AtBeginDocument{\let\originallabel\label} 

\theoremstyle{definition}
\newtheorem*{definition*}{Definition}

 \newtheorem{example}[theorem]{Example}

\theoremstyle{remark}
\newtheorem{remark}[theorem]{Remark}
\newtheorem*{acknowledgments}{Acknowledgments}

\numberwithin{equation}{section}

\begin{document}
\title{Mather Classes of Schubert Varieties via Small Resolutions}
\author{Minyoung Jeon}
\address{Department of mathematics, The Ohio State University, Columbus OH 43210, USA}
\email{jeon.163@buckeyemail.osu.edu}

%
\begin{abstract}
We express a Schubert expansion of the Chern-Mather class for Schubert varieties in the even orthogonal Grassmannian via integrals involving Pfaffians and pushforward of the small resolutions in the sense of Intersection Cohomology (IH) constructed by Sankaran and Vanchinathan, instead of the Nash blowup. The equivariant localization is employed to show the way of computing the integral. As a byproduct, we present the computations. For analogy and the completion of the method in ordinary Grassmannians, we also suggest Kazhdan-Lusztig classes associated to Schubert varieties in the Lagrangian and odd orthogonal Grassmannian. 
\end{abstract}
\maketitle
\setcounter{tocdepth}{2}
\tableofcontents

\section{Introduction}
The {\it Chern-Mather class}, defined by MacPherson \cite{M}, is one of the characteristic classes of singular varieties, along with the Chern-Schwartz-MacPherson class, the Fulton class and the Fulton-Johnson
class. These characteristic classes are significant in classical algebraic geometry, since they generalize the Chern class $c(TX)$ of a nonsingular variety $X$. For an irreducible, quasi-projective complex (possibly singular) variety $X$ embedded in a nonsingular variety $Y$, the Mather class $c_M(X)$ of $X$ is an element in the Chow group (or homology) $A_*(Y)$ and defined through the Nash blowup of $X$.

We consider Schubert varieties $\mathbb{S}(\underline{\alpha})$, which in most cases are singular varieties. In the case of the ordinary Grassmannians, so-called of Lie type A, Jones \cite{Benjamin} expressed the Chern-Mather classes of Schubert varieties by integrations over Zelevinsky's IH-small resolutions \cite{Z} (small resolutions in the sense of Intersection Cohomology), without the Nash blowup and computed the Mather classes by the use of equivariant localization. The method relies on the irreducibility of the characteristic cycle $CC(IC_{\mathbb{S}(\underline{\alpha})})$ associated to $\mathbb{S}(\underline{\alpha})$ in simply laced Lie types. 


 Sankaran and Vanchinathan \cite{SV} constructed IH-small resolutions of Bott-Samelson type
for Grassmannian Schubert varieties in types D and C. Our goal of this paper is to express the coefficients of the Schubert expansion for the Chern-Mather classes of Schubert varieties in even orthogonal Grassmannians $OG(n,\mathbb{C}^{2n})$ of Lie type D, as in the category of simply laced types, in terms of integrals involving Pfaffians along Sankaran and Vanchinathan's IH-small resolutions. When it comes to types B and C, the expressions we found from IH-small resolutions for Schubert varieties are for the {\it Kazhdan-Lusztig classes} investigated by Aluffi, Mihalcea, Schuermann and Su \cite{AMSS, AMSS20} as well as Mihalcea and Singh \cite{MS}. Essentially, they turn out that Jones' outcomes for the Chern-Mather classes coincide with the {Kazhdan-Lusztig classes} \cite[Page 15]{MS}. Since the Kazhdan-Lusztig class is defined regardless of the irreducibility of characteristic cycles, we further examine the Kazhdan-Lusztig classes of Schubert varieties in Lagrangian Grassmannians $LG(n,\mathbb{C}^{2n})$ of type C, and in the odd orthogonal Grassmannians $OG(n,\mathbb{C}^{2n+1})$ of type B, aiming to complete the direction of Zelevinsky's IH-small resolutions by Jones for classical Lie types.

Our main result describes the Chern-Mather classes of Sankaran and Vanchinathan's IH-small resolutions for Schubert varieties in the even orthogonal Grassmannians (type D), Lagrangian Grassmannians (type C), and the odd orthogonal Grassmannians (type B). Since the Chern-Mather class of a non-singular variety is the same as the total Chern class of its tangent bundle, we present the total Chern classes of them explicitly, using the universal subbundles as follows.

\begin{theorem}[Total Chern class of the IH-small resolutions]\label{Mainthm}
Let $Z_{\underline{\alpha}}\rightarrow \mathbb{S}(\underline{\alpha})$ be a IH-small resolution of a Schubert variety $\mathbb{S}(\underline{\alpha})$ in types D, C and B. Then the total Chern class of $Z_{\underline{\alpha}}$ is
\begin{enumerate}[label=(\roman*)]
\item (Type D) 
$c(TZ_{\underline{\alpha}})=\left(\prod_{i=1}^{d}c((\underline{U}_i/\underline{W}_i^L)^\vee\otimes(\underline{W}_i^R/\underline{U}_i))\cdot c\left(\mathrm{\wedge}^2(\underline{U}_{d+1}/\underline{W}_{d+1}^L)^\vee\right)\right)$
\item (Type C)
 $c(TZ_{\underline{\alpha}})=\left(\prod_{i=1}^{d}c((\underline{U}_i/\underline{W}_i^L)^\vee\otimes(\underline{W}_i^R/\underline{U}_i))\cdot c\left(\mathrm{Sym}^2(\underline{U}_{d+1}/\underline{W}_{d+1}^L)^\vee\right)\right)$
\item (Type B)
$c(TZ_{\underline{\alpha}})=\left(\prod_{i=1}^{d} c((\underline{U}_i/\underline{W}_i^L)^\vee\otimes(\underline{W}_i^R/\underline{U}_i))\cdot c((\underline{U}_{d+1}/\underline{W}_{d+1}^L)^\vee\otimes(\underline{U}_{d+1}^\perp/\underline{U}_{d+1}) )\right.$
  \begin{align*}
&\quad \left.\cdot c\left(\mathrm{\wedge}^2(\underline{U}_{d+1}/\underline{W}_{d+1}^L)^\vee\right)\right).
\end{align*}
\end{enumerate}

\end{theorem}

The above theorem is analogous to the total Chern classes of the resolutions $Z_{\underline{\alpha}}$ over Schubert varieties in ordinary Grassmannians $Gr(k,\mathbb{C}^n)$ of $k$-dimensional subspaces of a $n$-dimensional vector space over $\mathbb{C}$ by Jones \cite[Theorem 1.2.2]{Benjamin} as 
\begin{enumerate}[label=, align = left, leftmargin=5pt]
    \item \leavevmode\vspace{-\dimexpr\baselineskip + \topsep + \parsep}%
   \quad \quad\quad   {\it (Type A)}$\quad\quad\quad c_{M}(Z_{\underline{\alpha}})=\left(\prod_{i=1}^{d}c((\underline{U}_i/\underline{W}_i^L)^\vee\otimes(\underline{W}_i^R/\underline{U}_i))\right).\vspace{-2mm}$\\
   \\
\end{enumerate}
Because of the isomorphisms of the odd orthogonal Grassmannians $OG(n,\mathbb{C}^{2n+1})$ for type B and even orthogonal Grassmannians $OG'(n+1,\mathbb{C}^{2n+2})$ (or $OG''(n+1,\mathbb{C}^{2n+2}))$ for type D, the IH-small resolutions of Schubert varieties can be identified with the ones in type D. The isomorphisms allow us to be able to interpret any statements in type D as in type B. We refer reader to later sections (\S\ref{s4.1},\S\ref{s5.1},\S\ref{s5.2}) for undefined notations in Theorem \ref{Mainthm}.


The Kazhdan-Lusztig classes of Schubert varieties in isotropic or orthogonal Grassmannians can be signified as the pushforward of the total Chern classes $c(TZ_{\underline{\alpha}})$ of the tangent bundles of any IH-small resolutions $Z_{\underline{\alpha}}$, parallel to the Chern-Mather classes as the pushforward of $c(TZ_{\underline{\alpha}})$ \cite{Benjamin, AMSS20}. Since there is no explicit computation for the pushforward of the Chern classes of the tangent bundles of the IH-small resolutions of singularity except type A, we offer how to calculate them for the other classical types. 

The localization theorem for equivariant Chow groups \cite{B} is employed to compute the pushforward to the corresponding ambient Grassmannians of Schubert varieties. 
We adapt the work by Pragacz \citelist{\cite{P1}\cite{P2}} who showed {\it Pfaffian} formulas for the (co)homology classes of Schubert varieties in Grassmannians of isotropic subspaces of a vector space equipped with a nondegenerate quadratic or symmetric form, commonly known as Schur $\widetilde{P}$ or $\widetilde{Q}$ functions in algebraic combinatorics, to find the coefficients $\gamma_{\underline{\alpha}, \underline{\beta}}\in\mathbb{Z}$ of the Schubert classes $\left[\mathbb{S}(\underline{\beta})\right]$ in $\pi_*c_{M}(Z_{\underline{\alpha}})$. Here $\pi:Z_{\underline{\alpha}}\rightarrow \mathbb{S}(\underline{\alpha})$ is the IH-small resolution. 

We obtain the following statements from the Bott Residue formula (Theorem \ref{Bott}). Our formulas reduced to explicit computations of $\mathbb{C}^*$-equivariant Chern classes $c^{\mathbb{C}^*}(\underline{E})$ and the Pfaffians $\widetilde{P}^{\mathbb{C}^*}_\lambda(\underline{E})$ or $\widetilde{Q}^{\mathbb{C}^*}_\lambda(\underline{E})$ for some ${\mathbb{C}^*}$-equivariant vector bundles $\underline{E}$ over a nonsingular variety and partitions $\lambda=(\lambda_1,\ldots,\lambda_s)$. We define $|\lambda|:=\lambda_1+\cdots+\lambda_s$. Let $F$ be any connected components of $Z_{\underline{\alpha}}$ and $\pi_{F*}:A^T_*F\rightarrow R_T$ is the push-forward map induced by the map $\pi_F$ from $F$ to a point where $R_T$ is the $T$-equivariant Chow ring of a point.
\begin{theorem}[Coefficients of Schubert classes]\label{thm:main2}
\begin{enumerate}[label=(\roman*)]
\item[]
\item (Type D and B) Let $Z_{\underline{\alpha}}\rightarrow \mathbb{S}(\underline{\alpha})$ be a IH-small resolution for a Schubert variety $\mathbb{S}(\underline{\alpha})$ in the even orthogonal Grassmannian $OG'(n,\mathbb{C}^{2n})$ (resp. $OG''(n,\mathbb{C}^{2n}))$ or the odd orthogonal Grassmannian $OG(n-1,\mathbb{C}^{2n-1})$. Then the constant $\gamma_{\underline{\alpha}, \underline{\beta}}$ is the integration
 \[
 \gamma_{\underline{\alpha}, \underline{\beta}}=\displaystyle\sum_{F\in Z^{\mathbb{C}^*}_{\underline{\alpha}}}\pi_{F*}\left(\dfrac{c_k^{\mathbb{C}^*}(TZ_{\underline{\alpha}}|_F)\cdot \widetilde{P}^{\mathbb{C}^*}_{\rho(n-1)\backslash\underline{\beta}}(\underline{U}^\vee|_F)\cap \left[F\right]_{\mathbb{C}^*}}{c_{d}^{\mathbb{C}^*}(TZ_{\underline{\alpha}})}\right)
 \]
where $d=\mathrm{dim}(Z_{\underline{\alpha}})$, $k=d-|\rho(n-1)\backslash\underline{\beta}|$, and $\underline{U}$ is the universal tautological subbundle on $OG'(n,\mathbb{C}^{2n})$ (resp. $OG''(n,\mathbb{C}^{2n}))$ or $OG(n-1,\mathbb{C}^{2n-1})$.
\item (Type C) Let $Z_{\underline{\alpha}}\rightarrow \mathbb{S}(\underline{\alpha})$ be a IH-small resolution for a Schubert variety $\mathbb{S}(\underline{\alpha})$ in the Lagrangian Grassmannian $LG(n,\mathbb{C}^{2n})$. The constant $\gamma_{\underline{\alpha}, \underline{\beta}}$ is given by
\[
\gamma_{\underline{\alpha}, \underline{\beta}}=\displaystyle\sum_{F\in Z^{\mathbb{C}^*}_{\underline{\alpha}}}\pi_{F*}\left(\dfrac{c_k^{\mathbb{C}^*}(TZ_{\underline{\alpha}}|_F)\cdot \widetilde{Q}^{\mathbb{C}^*}_{\rho(n)\backslash\underline{\beta}}(\underline{U}^\vee|_F)\cap\left[F\right]_{\mathbb{C}^*}}{c_{d}^{\mathbb{C}^*}(TZ_{\underline{\alpha}})}\right)
\]
where $d=\mathrm{dim}(Z_{\underline{\alpha}})$, $k=d-|\rho(n-1)\backslash\underline{\beta}|$, and $\underline{U}$ is the universal tautological subbundle on $LG(n,\mathbb{C}^{2n})$.
\end{enumerate}
\end{theorem}
%
%
In Theorem \ref{thm:main2}, the Pfaffians $\widetilde{P}^{\mathbb{C}^*}_\lambda(\underline{U})$ or $\widetilde{Q}^{\mathbb{C}^*}_\lambda(\underline{U})$ are the square root of the determinant of a skew-symmetric matrix in $c^{\mathbb{C}^*}(\underline{U})$. The exact definitions of these Pfaffians will be discussed in \S\ref{sec4}, pg. ~\pageref{page:Pfa}-~\pageref{page:Pfa1}, and some useful properties of Pfaffians can be found in \cite[Appendix D]{FP}. The Chern classes $c_i^{\mathbb{C}^*}(TZ_{\underline{\alpha}})$ and Pfaffians $\widetilde{P}^{\mathbb{C}^*}_{\rho(n-1)\backslash\underline{\beta}}(\underline{U}^\vee)$ and $\widetilde{Q}^{\mathbb{C}^*}_{\rho(n)\backslash\underline{\beta}}(\underline{U}^\vee)$ can be computed by formulas in the (equivariant version of) intersection theory, for instance \cite[\S3, A.9]{Fulton} and \cite[Lem. 5.1.4]{Benjamin}. The examples of these computations are included in Sections \ref{ExmD} and \ref{s5.1}. These formulas for the coefficients $\gamma_{\underline{\alpha},\underline{\beta}}$ are analogous to the one for type A in \cite{Benjamin} given by

\begin{enumerate}[label=, align = left, leftmargin=5pt]
    \item \leavevmode\vspace{-\dimexpr\baselineskip + \topsep + \parsep}%
   \quad \quad\quad   {\it (Type A)}$\quad\quad\quad\quad\gamma_{\underline{\alpha}, \underline{\beta}}=\displaystyle\sum_{F\in Z^{\mathbb{C}^*}_{\underline{\alpha}}}\pi_{F*}\left(\dfrac{c_k^{\mathbb{C}^*}(TZ_{\underline{\alpha}}|_F)\cdot s^{\mathbb{C}^*}_{\underline{\beta}^\vee}(\underline{U}^\vee|_F)\cap\left[F\right]_{\mathbb{C}^*}}{c_{d}^{\mathbb{C}^*}(TZ_{\underline{\alpha}})}\right)
\vspace{-2mm}$\\
   \\
\end{enumerate}
where $s^{\mathbb{C}^*}_\lambda(\underline{U}^\vee)$ is the Schur determinant of the $s$ by $s$ matrix whose $(i,j)$ entry is $c_{\lambda_i+j-i}^{\mathbb{C}^*}(\underline{U}^\vee)$,  $d=\mathrm{dim}(Z_{\underline{\alpha}})$, $k=d-|\beta^\vee|$, and $\underline{U}$ is the universal tautological subbundle on $Gr(k,\mathbb{C}^{n})$.

In this manner, we eventually provide general explicit combinatorial recipes calculating the Chern-Mather classes $c_M(\mathbb{S}(\underline{\alpha}))$ of Schubert varieties in the orthogonal Grassmannians, which partially recovers consequences in \cite{MS}, and Kazhdan-Lusztig classes $KL(\mathbb{S}(\underline{\alpha}))$ of Schubert varieties in Lagrangian Grassmannians, in respect of the (homology) class of Schubert varieties $\mathbb{S}(\underline{\beta})\subseteq \mathbb{S}(\underline{\alpha})$ for some sequences $\underline{\alpha}$ and $\underline{\beta}$.

The key ingredient of our proof is the existence of IH-small resolutions for Schubert varieties. N. Perrin \cite{Per} classified all minuscule Schubert varieties that admit IH-small resolutions. It would be interesting to compute the Chern-Mather classes or Kazhdan-Lusztig classes of minuscule Schubert varieties via the small resolutions of Perrin. Beyond minuscule (or cominuscule) Schubert varieties in $G/P$, Larson \cite[Section 4]{L} made IH-small resolutions for Schubert varieties associated to certain Weyl group elements from IH-small resolutions for the other Schubert varieties in $G/B$. It would also be of interest to apply our methods to Larson's resolutions, expanding the computations of Mather classes to special Schubert varieties in $G/B$.

Beside our approach by IH-small resolutions and the advent of Pfaffians for types D, B and C, Mihalcea and Singh studied Mather classes from resolutions for the conormal spaces of cominuscule Schubert varieties in the equivariant setting \cite{MS}. We also refer to \cite{RP,Zhang} for degeneracy loci of several types.

As for the Nash-blowup, Richmond, Slofstra and Woo computed the Nash-blowup of cominuscule Schubert varieties and gave explicit correspondences between the Nash-blowup and the Zelevinsky's IH-small resolutions \cite{RSW}. One may determine the Mather classes from their Nash-blowup of Schubert varieties in all cominuscule homogeneous spaces by the original definition.

%

\begin{acknowledgments}
The author wishes to thank David Anderson for invaluable suggestions and a lot of thorough reading of preliminary versions of this paper. We also wish to express our gratitude to Leonardo Mihalcea for his insightful comments to enhance the accuracy of the original manuscript and sharing their work with the author. MJ was partially supported by NSF CAREER DMS-1945212 from her advisor David Anderson. Lastly, we are very grateful to Xiping Zhang and anonymous referee for the careful reading of this manuscript, helpful suggestions and valuable comments. 
\end{acknowledgments}

\section{Chern-Mather classes and Kazhdan-Lusztig classes}\label{sec2}
 In this section we review some basic facts on Chern-Mather classes of certain complex algebraic varieties and Kazhdan-Lusztig classes of Schubert varieties in $G/P$ taking resolutions into account. Main references for this section are \cite{M} and \cite[\S 2-\S 3]{Benjamin}, but we occasionally use  \cite{GH,BS,MS}.

\subsection{Mather classes by resolution of singularities} \label{sec2.1}
Let $M$ be a smooth algebraic variety over $\mathbb{C}$ and $X$ an irreducible closed subvariety of dimension $n$ in $M$. Let $Gr(n,TM)\rightarrow M$ be the Grassmannian bundle over $M$. The {\it Gauss map} $\mathcal{G}:X\dashrightarrow Gr(n, TM)$ is a rational morphism that assigns a smooth point $x$ to the tangent space $T_xX$ of $X$ at the point $x$. The {\it Nash blowup} $\widetilde{X}$ of $X$ is the closure of the image of $\mathcal{G}$, and the tautological {\it Nash tangent} bundle $\mathcal{T}$ is the restriction of the tautological sub-bundle of $Gr(n, TM)$ to the Nash blowup $\widetilde{X}$.

Provided the Nash blowup $\nu:\widetilde{X}\rightarrow X$, the {\it Chern-Mather} class $c_M(X)$ of $X$ is defined to be 
$$c_M(X):=\nu_*\left(c(\mathcal{T})\cap [ \widetilde{X}]\right)\in A_*(X).$$
If $X$ is smooth, the tautological Nash tangent bundle $\mathcal{T}$ becomes its tangent bundle $TX$ so that the Chern-Mather class $c_M(X)$ is equal to the total homology Chern class of $X$, i.e.,
$$c_M(X)=c(TX)\cap[X].$$

One can use the functoriality of the Chern-Schwartz-MacPherson class along with resolution of singularities to compute the Mather class in the place of the Nash-blowup that does not have the functorial property. We recall the definitions and notions of the local Euler obstruction and the Chern-Schwartz-MacPherson class before we express the Mather class as Chern-Schwartz-MacPherson classes.

Let $\mathcal{X}$ be a proper subvariety of a (quasi-projective) complex variety $\mathcal{Y}$. We denote by $Bl:\mathcal{Y}'\rightarrow \mathcal{Y}$ the blowup of $\mathcal{Y}$ along $\mathcal{X}$ with the exceptional divisor $\mathcal{E}$ of $Bl$. The Segre class $s(\mathcal{X},\mathcal{Y})$ of $\mathcal{X}$ and $\mathcal{Y}$ is given by 
$$s(\mathcal{X},\mathcal{Y})=(Bl|_{\mathcal{E}})_*\sum_{j\geq 1}(-1)^{j-1}\left[\mathcal{E}^j\right]\in A_*(\mathcal{X}),$$
where $\left[\mathcal{E}^{j}\right]:=c_1(\mathcal{O}_{\mathcal{Y}'}(\mathcal{E}))^{j-1}\cap\left[\mathcal{E}\right]$.
Given a fixed point $p$ in $X$, the {\it local Euler Obstruction} of $X$ at $p$ is the number
\begin{equation}\label{eqn:obs}
\mathrm{Eu}_X(p):=\int_{\nu^{-1}(p)}c(\mathcal{T}|_{\nu^{-1}(p)})\cap s(\nu^{-1}(p),\widetilde{X}) 
\end{equation}
by Gonz\'{a}lez-Sprinberg and Verdier \cite{Gon} where $s(\nu^{-1}(p),\widetilde{X})$ is the Segre class of $\nu^{-1}(p)$ in $\widetilde{X}$. In fact the original definition of the local Euler Obstruction is defined topologically by MacPherson. We note that $\mathrm{Eu}_X(p)=1$ if a point $p$ is smooth in $X$. 


Let $F_*(X)$ be the group of constructible functions on $X$ and $\mathbbm{1}_{W}$ the characteristic function of $W$ for a closed subset $W\subset X$. The elements of $F_*(X)$ are expressed as a finite sum ${\sum_i} a_i \mathbbm{1}_{W_i}$ for $a_i\in \mathbb{Z}$ and closed subsets $W_i\subset X$. 
We observe that the local Euler obstruction $\mathrm{Eu}_X:X\rightarrow \mathbb{Z}$ is constructible with respect to a Whitney stratification of $X$. Namely the function $\mathrm{Eu}(\mathbbm{1}_{W})(p)$ assigning $\mathrm{Eu}_{W}(p)$ if $p\in W$ and $0$ otherwise can be extended linearly as a basis of $F_*(X)$. 

Taking for granted that $f:X\rightarrow Y$ is a proper morphism, the pushforward $f_*:F_*(X)\rightarrow F_*(Y)$ induced by $f$ is defined to be $f_*(\mathbbm{1}_W)(p)=\chi(f^{-1}(p)\cap W)$ where $\chi$ is the topological Euler characteristic. The main result by MacPherson is the existence of the unique natural transformation $c_*:F_*\rightarrow A_*$ in the sense that firstly, 
\begin{equation}\label{eqn:sm}
c_*(\mathbbm{1}_X)=c(TX)\cap \left[X\right]
\end{equation}
if $X$ is smooth and secondly, the following functoriality holds: for any proper morphism $f:Y\rightarrow X$, the diagram 
\[
\begin{tikzcd}
F_*(Y)\arrow[r, "c_*"]\arrow[d, "f_*"]& A_*(Y)\arrow[d, "f_*"] \\
F_*(X)\arrow[r,  "c_*"]&  A_*(X)
\end{tikzcd}
\]
commutes. 
Let $W$ be closed in $X$. The {\it Chern-Schwartz-MacPherson} (CSM) class $c_{SM}(W)$ of $W$ is defined by the image of the characteristic function $\mathbbm{1}_W$ under the transformation $c_*$ as
$$c_{SM}(W):=c_*(\mathbbm{1}_W).$$
We can also define the CSM class for locally closed subsets $S$ of any fixed variety $M$. That is, if $S=X\backslash Y$ for closed subsets $X$ and $Y$ in $M$, the CSM class of $S$ can be attained by $$c_{SM}(S)=c_{SM}(X)-c_{SM}(Y).$$
In regard to a proper morphism $i:M\rightarrow N$ satisfying that $i_*$ is injective, we may write $c_{SM}(S)$ for $i_*c_{SM}(S)\in A_*(N)$. In fact, the injectivity of $i_*$ may not be necessary if $\{\mathcal{W}_i\}$ are the Whitney stratification of $N$ and $M=\overline{\mathcal{W}_i}$ is the closure of $\mathcal{W}_i$ with a locally closed set $S$ in $M$. In this case, since $\mathbbm{1}_S$ is also constructible in $N$, $c_{SM}(S)$ can be viewed as an element in $A_*(N)$.

The local Euler obstruction $\mathrm{Eu}_X\in F_*(X)$ on $X$ can be viewed as a finite sum $\sum_i e_i\mathbbm{1}_{\overline{\mathcal{W}_i}}$ of characteristic functions for any stratification $\{\mathcal{W}_i\}$ of $X$ where by definition $e_i$ is the local Euler obstruction $\mathrm{Eu}_{\overline{\mathcal{W}_i}}(p)$ of $\mathcal{W}_i$ at any point $p\in\overline{\mathcal{W}_i}$. Since the Mather class $c_{M}(X)$ can be seen as to the transformation, i.e., $c_{M}(X)= c_*(\mathrm{Eu}_X)$, we have the Mather class  
\begin{equation}\label{eqn:mather}
c_{M}(X)=c_*(\sum_i e_i\mathbbm{1}_{\overline{\mathcal{W}_i}})=\sum_i e_i c_*(\mathbbm{1}_{\overline{\mathcal{W}_i}})=\sum e_i c_{SM}(\overline{\mathcal{W}_i})
\end{equation}
of $X$ in connection with CSM classes.

We consider a Whitney stratification $\{\mathcal{W}_i\}_{i\in I}$ of a variety $M=\cup_{i\in I}\mathcal{W}_i$ for an total ordering index set $I$ such that 
\begin{enumerate}
\item $\mathcal{W}_i\subset \overline{\mathcal{W}}_j$ if and only if $i\leq j$ for $i,j\in I$, and
\item $X=\overline{\mathcal{W}}_{i_0}$ for some $i_0\in I$.
\end{enumerate}
Let $\pi:Z\rightarrow X$ be a resolution of singularities of $X$. We assume the existence of resolutions of singularities $\pi_i:Z_i\rightarrow \overline{\mathcal{W}}_i$ on each stratum $\overline{\mathcal{W}}_i$ for $i\in I$ and the restriction of the resolutions $\pi_i$ on any stratum $\mathcal{W}_j\subset \overline{\mathcal{W}}_i$ as a fiber bundle. Let $d_{i,j}$ denote the topological Euler characteristic $\chi(\pi_i^{-1}(p))$ of the fiber of the resolution $\pi$ over any point $p$ in some strata $\mathcal{W}_j$ for $i,j\in I$. Then we land at
 \begin{equation} \label{eq}
 (\pi_i)_*c_{M}(Z_i)\labelrel={eq:a}(\pi_i)_*c_{SM}(Z_i)\labelrel={eq:b}(\pi_i)_*c_{*}(\mathbbm{1}_{Z_i})\labelrel={eq:c}c_{*}(\pi_i)_*(\mathbbm{1}_{Z_i})=\sum_{j\leq i}d_{i,j}c_*(\mathbbm{1}_{\mathcal{W}_j}).
 \end{equation}
Here the first equality ~\eqref{eq:a} comes from the fact that the Chern-Mather class coincides with CSM class if the variety is nonsingular, the second equality ~\eqref{eq:b} is by definition and the third one ~\eqref{eq:c} by the naturality of the MacPherson transformation $c_*$.

\subsection{Mather classes via IH-small resolutions}

Let $X$ be an irreducible subvariety of a smooth complex algebraic variety $M$. By \cite[\S 1.1]{GM}, $X$ admits a stratification, so that we can define the {\it intersection cohomology (IC) sheaf} of $X$ denoted by $\mathcal{IC}_X^\bullet$ \cite[Intro.]{GM1}. The $\mathcal{IC}$-sheaf of $X$ is constructible with respect to any Whitney stratifications of $M$ \cite[\S 3.1]{Benjamin}, and it is a (middle perversity) perverse sheaf on $M$. 

A resolution $\pi:Z\rightarrow X$ is {\it IH-small} (in the intersection cohomolgy sense) if 
$$\mathrm{codim}\;\{p\in X\;|\;\mathrm{dim}\;\pi^{-1}(p)\geq i\}>2i$$
 for all $i>0$. This resolution is referred to by Totaro \cite{T} as the IH-small resolution whereas Goresky and MacPherson \cite{GM1} originally calls it the small resolution. A conceivable reason to adapt name for $\pi$ as the IH-small resolution likely stemmed from the property that the intersection homology of $X$ is isomorphic to the ordinary cohomology of $Z$. 
 
 We use $H^\bullet(X;\mathbb{C})$ to denote the ordinary cohomology of $X$ with complex coefficients. Let $D^b(X)$ be the constructible {\it derived category} on a (quasi-projective) complex variety $X$ and $Rf_*:D^b(X)\rightarrow D^b(Y)$ be the right derived functor of the direct image functor for $f:X\rightarrow Y$ of (quasi-projective) complex varieties. Let $\mathbb{C}_Y$ be the constant sheaf in degree zero having stalk $\mathbb{C}$ at all points of $Y$. If $M$ is a smooth complex algebraic variety of the dimension $m$, then we have $\mathcal{IC}_M^\bullet=\mathbb{C}_M[2m]$ where $\mathcal{F}[n]^i$ indicates the complex $\mathcal{F}^{i+n}$ of sheaves $\mathcal{F}^\bullet$. One may refer to \cite{GM1} for these notations.
 The topological relation between the locus $Z$ and the base variety $X$ can be found in the following proposition by Goresky-MacPherson.

\begin{proposition}[{\cite[\S 6.2]{GM1}}]\label{Euler}
Let $X$ be a $d$-dimensional irreducible complex algebraic variety. If $\pi:Z\rightarrow X$ is a IH-small resolution of singularities, then $R\pi_*\mathbb{C}_Z[2d]\cong \mathcal{IC}_X^\bullet$. In particular for a point $p\in X$,
$$\chi_p(\mathcal{IC}_X^\bullet)=\sum_i(-1)^i\mathrm{dim} \;H^i(\pi^{-1}(p);\mathbb{C})=\chi(\pi^{-1}(p))$$
where $\chi_p(\mathcal{IC}_X^\bullet)$ denotes the stalk Euler characteristic of $\mathcal{IC}$-sheaf at the point $p$.
\end{proposition}

Let us take a smooth complex variety $M$ equipped with a Whitney stratification $M=\cup_{i\in I}\mathcal{W}_i$ and an irreducible variety $X\subset M$. For the $\mathcal{IC}$-sheaf $\mathcal{IC}_X^\bullet$ of $X$, we have a corresponding cycle $\mathcal{CC}(\mathcal{IC}_X^\bullet)$. This cycle is called the {\it characteristic cycle} of the $\mathcal{IC}$-sheaf of $X$ in the cotangent bundle $T^*M$. As $\mathcal{IC}_X^\bullet$ is constructible with respect to the stratification $\{\mathcal{W}_i\}$ of $M$, the characteristic cycle of $\mathcal{IC}$-sheaf may be written as a (conical) Lagrangian cycle
$$\mathcal{CC}(\mathcal{IC}_X^\bullet)=\sum_{i\in I}\gamma_i(\mathcal{IC}_X^\bullet)\cdot\left[\overline{T^*_{\mathcal{W}_i}M}\right],$$
an element in the free abelian group generated by the {\it conormal cycles} $\left[\overline{T^*_{\mathcal{W}_i}M}\right]$ of $\mathcal{W}_i$ (cf. \cite{BDK}). Here the integer $\gamma_i(\mathcal{IC}_X^\bullet)$ is the {\it microlocal multiplicity} of $\mathcal{IC}_X^\bullet$ along $\mathcal{W}_i$. The cycle may be discussed in the perspective of the category of holonomic $\mathcal{D}_M$-modules \cite[\S 5.3]{D} or topological link spaces \cite[Section 4.1]{D}. For the IC-sheaf $\mathcal{IC}_X^\bullet$ on $M$, the local Euler obstruction along the $j$-th stratum $\mathcal{W}_j$ in the closure of $\mathcal{W}_i$ can be related to the microlocal multiplicity of $\mathcal{IC}^\bullet_X$ and the stalk Euler characteristic $\chi_i(\mathcal{IC}_X^\bullet)=\chi_{p_i}(\mathcal{IC}_X^\bullet)$ for $p_i\in \mathcal{W}_i$ as follows. (cf. \cite[Theorem 3]{D01}, \cite[Theorem 6.3.1]{K}.)

\begin{theorem}[Microlocal index formula for $\mathcal{IC}$-sheaf]
For any $i\in I$ 
$$\chi_j(\mathcal{IC}_X^\bullet)=\sum_{i\in I} (-1)^{n_i} \mathrm{Eu}_{\overline{\mathcal{W}}_i}(\mathcal{W}_j)\cdot \gamma_i(\mathcal{IC}_X^\bullet)$$
where $n_i$ is the dimension of $\mathcal{W}_i$.
\end{theorem}
We mainly focus on the case where the variety $X$ admits a IH-small resolution $\pi:Z\rightarrow X$, especially on the subject of Chern-Mather classes of $X$. 
We say that the characteristic cycle of $\mathcal{IC}_X^\bullet$ is irreducible if
\[
CC(\mathcal{IC}_X^\bullet)=\left[\overline{T^*_{\mathcal{W}_{i_0}}M}\right].
\]
With an irreducible characteristic cycle of $\mathcal{IC}_X^\bullet$, the proposition below tells us a direct connection between the local Euler obstruction and the topological Euler characteristic $d_{i,j}$ \eqref{eq} of the fiber of the resolution.

\begin{proposition}[{\cite[Proposition 3.2.3]{Benjamin}}]\label{prop:2.3}
Let $\pi:Z\rightarrow X$ be a IH-small resolution of singularities and the characteristic cycle of $\mathcal{IC}_X^\bullet$ is irreducible. Then we have
$$d_{i_0,j}=\mathrm{Eu}_X(p_j)$$
for $p_j\in \mathcal{W}_j$.
\end{proposition}

Under the assumption that the characteristic cycle of $\mathcal{IC}_X^\bullet$ is irreducible, the Chern-Mather class of $X$ can be achieved by a IH-small resolution of $X$ as follows. 
\begin{theorem}[{\cite[Theorem 3.3.1]{Benjamin}}]\label{thm:CM}
The Chern-Mather class of $X$ is the pushforward of the total Chern class of the variety $Z$ as
$$c_{M}(X)=\pi_*c(TZ).$$
\end{theorem}

\begin{proof}
Since $Z=Z_{i_0}$ and $X=X_{i_0}$, we have 
\begin{align*}
c_{M}(X)&\labelrel={eqn;sm a}\sum \mathrm{Eu}_{\overline{\mathcal{W}_j}}(p_j)c_*(\mathbbm{1}_{\overline{\mathcal{W}_j}})\labelrel={eqn;sm b}\sum_{j\leq i_0}d_{i_0,j}c_*(\mathbbm{1}_{\overline{\mathcal{W}_j}})\labelrel={eqn;sm c}\pi_*c_{SM}(Z)\\
&\labelrel={eqn;sm d}\pi_*c(TZ)
\end{align*}
where the equality ~\eqref{eqn;sm a} follows from \eqref{eqn:mather}, the second equality ~\eqref{eqn;sm b} by Proposition \ref{prop:2.3} and the third equality ~\eqref{eqn;sm c} by the equation \eqref{eq}. Lastly the smoothness of $Z$ guarantees the equality ~\eqref{eqn;sm d} by means of \eqref{eqn:sm}.
\end{proof}

\subsection{Kazhdan-Lusztig classes of Schubert varieties}

Let $G$ be a complex semisimple Lie group, $P$ a standard parabolic subgroup, $B$ a Borel subgroup, and $B^{-}$ the opposite Borel subgroup, with a maximal torus $T=B\cap B^{-1}$ such that $T\subset B\subset P$. Let $W:=N(T)/T$ be the Weyl group of $G$ where $N(T)$ is the normalizer of $T$. To be specific, our attention focuses on a classical group $G$, which is $SL(n)$ for type A, $Sp({2n})$ for type C, $SO(2n+1)$ for type B and $SO(2n)$ for type D with their Weyl group $W_n^A, W_{n}^C, W_{n}^B$ and $W_{n}^D$. 

For a classical group $G$, we denote by $G/P$ the generalized flag manifold and $W_P\subset W$ the Weyl group of $P$. Let $W^P$ be the set of minimal representatives of the coset $W/W_P$, so that it has a role of an index set for the $T$-fixed points $(G/P)^T$. It is notorious that there is a one to one correspondence between an element $w_{\underline{\alpha}}\in W^P$ and a partition $\underline{\alpha}=(1\leq \alpha_1\leq\cdots\leq\alpha_s)$ for some $s$: for instance in type A, the Weyl group $W^A_n$ is identified with the symmetric group $S_n$ so that an element $w_{\underline{\alpha}}\in W^P\subset S_n$ defines a partition $\underline{\alpha}$ by setting $\alpha_{s+1-k}=n-w(k)+k$ and vice versa. So we may use $\underline{\alpha}$ for the element $w_{\underline{\alpha}}\in W^P$ by abuse of notation. 

Let $\mathbb{S}(\underline{\alpha})^\circ:=Bw_{\underline{\alpha}}P/P$ be a Schubert cell in $G/P$ for $w_{\underline{\alpha}}\in W^P$. For the length function $\ell:W\rightarrow \mathbb{N}$, the Schubert variety $\mathbb{S}(\underline{\alpha})$ whose dimension is $\ell(w_{\underline{\alpha}})$ is the $B$-orbit closure $\overline{Bw_{\underline{\alpha}}P/P}$ of a $T$-fixed point $p_{\underline{\alpha}}:=w_{\underline{\alpha}}P/P$. The Schubert variety associated to the longest element $w_\circ\in W^P$ can be treated as a homogeneous space $G/P$ that possesses a Whitney stratification by its sub-Schubert varieties $\mathbb{S}(\underline{\beta})$ for $w_{\underline{\beta}}\leq w_\circ$ in Bruhat order.

The {\it Kazhdan-Lusztig (KL) class} of a Schubert variety $\mathbb{S}(\underline{\alpha})$ in $G/P$ is defined by 
\begin{equation}\label{eq:KLspan}
KL(\mathbb{S}(\underline{\alpha}))=\sum_{\underline{\beta}}P_{\underline{\alpha},\underline{\beta}}(1) c_{SM}(\mathbb{S}({\underline{\beta}})^\circ)
\end{equation}
where $P_{\underline{\alpha},\underline{\beta}}(q)$ is the Kazhdan-Lusztig polynomial. The KL class of a Schubert variety turns out to be the pushforward of the total Chern class of the IH-small resolution.

\begin{theorem}[{\cite[Section 6, pag. 10]{AMSS}}]\label{thm:KL}
Let $\mathbb{S}(\underline{\alpha})$ be a Schubert variety and $\pi:Z\rightarrow \mathbb{S}(\underline{\alpha})$ a IH-small resolution of singularities over $\mathbb{S}(\underline{\alpha})$. Then
$$KL(\mathbb{S}(\underline{\alpha}))=\pi_*c(TZ).$$
\end{theorem}
\begin{proof}
We know from \cite[Theorem 12.2.5]{HTT} that 
$$P_{\underline{\alpha},\underline{\beta}}(1)=\sum_j(-1)^j\mathrm{dim}\; \mathcal{H}^j(\mathcal{IC}(\mathbb{S}(\underline{\alpha})))_{p_{\underline{\beta}}},$$
where $\mathcal{H}^j(\mathcal{IC}(\mathbb{S}(\underline{\alpha})))_{p_{\underline{\beta}}}$ indicates the stalk of the $j$-th cohomology sheaf $\mathcal{H}^j(\mathcal{IC}(\mathbb{S}(\underline{\alpha})))$ of the $\mathcal{IC}$-sheaf of the Schubert variety $\mathbb{S}(\underline{\alpha})$ at a $T$-fixed point $p_{\underline{\beta}}$. Notably, $\mathcal{H}^j(\mathcal{IC}(\mathbb{S}(\underline{\alpha})))_{p_{\underline{\beta}}}$ vanishes for odd number $j$.
Proposition \ref{Euler} and \cite[Proposition 1]{Z} yield that the stalk $\mathcal{H}^j(\mathbb{S}(\underline{\alpha}))_{p_{\underline{\beta}}}$ is isomorphic to the $j$-th cohomology $H^j(\pi^{-1}(p_{\underline{\beta}});\mathbb{C})$ to get 
$$\chi_{p_{\underline{\beta}}}(\mathcal{IC}_{\mathbb{S}(\underline{\alpha})}^\bullet)=P_{\underline{\alpha},\underline{\beta}}(1).$$
Let $d_{\underline{\alpha},\underline{\beta}}=\chi(\pi^{-1}(p_{\underline{\beta}}))$. Owing to \eqref{eq} and Proposition \ref{Euler} we have 
\begin{align*}
\pi_*c(TZ)&=\sum_{\underline{\beta}\leq\underline{\alpha}}d_{\underline{\alpha},\underline{\beta}}c_*(\mathbbm{1}_{\mathbb{S}(\underline{\beta})^\circ})=\sum_{\underline{\beta}\leq \underline{\alpha}}P_{\underline{\alpha},\underline{\beta}}(1)c_{SM}(\mathbb{S}(\underline{\beta})^\circ)\\
&=KL(\mathbb{S}(\underline{\alpha})).\qedhere
\end{align*}
\end{proof}

We observe that the characteristic cycle $\mathcal{CC}(\mathcal{IC}_X^\bullet)$ of the $\mathcal{IC}$-sheaf is irreducible if and only if the Kazhdan-Lusztig polynomial evaluated at $q=1$ gives the local Euler obstruction
$$P_{\underline{\alpha},\underline{\beta}}(1)=Eu_{\mathbb{S}({\underline{\alpha}})}(p_{\underline{\beta}})$$
for $\underline{\beta}\in W^P$, which entails $KL(\mathbb{S}(\underline{\alpha}))=c_{M}(\mathbb{S}(\underline{\alpha}))$.

\section{IH-small resolutions of Schubert varieties in the orthogonal Grassmannian $OG(n,\mathbb{C}^{2n})$}\label{sec3}

Throughout this section, we largely refer to \cite{Benjamin} for some notations and structures and \cite{SV} for Sankaran and Vanchinathan's IH-small resolution for Schubert varieties inside the even orthogonal Grassmannians of maximal isotropic subspaces.

\subsection{Schubert varieties in Grassmannians of type D}
Let $G=SO(2n)$ be the special orthogonal group in dimension $2n$ over $\mathbb{C}$. Let $V$ be a vector space of rank $2n$ over $\mathbb{C}$, equipped with a nondegenerate quadratic form. An {\it isotropic} subspace $L$ of $V$ is a subspace of $V$ such that $L$ vanishes on the form, in other words, $L\subset L^\perp$ with respect to the symmetric form associated to the quadratic form. The projective homogeneous space $G/P$ can be characterized as the even orthogonal Grassmannian $OG(n,V)$ that parametrizes the maximal (rank $n$) isotropic subspaces of $V$. 
We consider a complete flag of isotropic subspaces 
$$0\subseteq V_1\subseteq \cdots\subseteq  V_n\subseteq (V_{n-1})^\perp\subseteq \cdots \subseteq (V_1)^\perp\subseteq V$$
of $V$ where $(V_i)^\perp=V_{2n-i}$ and the rank of $V_i$ is $i$. 
Let $\underline{\alpha}=(1\leq \alpha_1<\alpha_2<\cdots<\alpha_s\leq n)$ be a sequence of positive integers such that $n-s$ is even.

 For a fixed flag $V_{\alpha_1}\subseteq\cdots\subseteq V_{\alpha_s}\subset V$ in the partial flag $Fl^D(\underline{\alpha};V)$ of isotropic subspaces, the Schubert variety $\mathbb{S}(\underline{\alpha})$ is given by the closure of the locus called {\it Schubert cell}
\begin{equation*}\label{typeD:sch}
\mathbb{S}(\underline{\alpha})^\circ=\{L\;|\;\mathrm{dim}(L\cap V_{\alpha_i})= i\;\text{for all}\;1\leq  i\leq s\}\subset OG'(n, V)\;(\text{resp.}\;OG''(n,V))
\end{equation*}
 associated to $\underline{\alpha}$.
The dimension of the Schubert variety is $\sum_{i\leq s}\alpha_i +n(n-s)-\dfrac{1}{2}n(n+1).$
In principle, the rank conditions may contain the case of $V_{\alpha_s}=V_n$ to satisfy
 \begin{equation*}\label{parity}
 \mathrm{dim}(L\cap V_n)\equiv n\;\text{(mod $2$)}\; \text{(resp.}\; \mathrm{dim}(L\cap V_n)\equiv n+1\;\text{(mod $2$)}).
 \end{equation*}
  We say that the maximal isotropic subspace $L$ for the first case is in the same family as $V_n$ and the later in the opposite family.  
Moreover, there is another Schubert variety $\mathbb{S}(\underline{\beta})$ associated to a sequence $\underline{\beta}=(1\leq \beta_1<\beta_2<\cdots<\beta_r\leq n)$ so that $\mathbb{S}(\underline{\beta})\subseteq\mathbb{S}(\underline{\alpha})$ is if $s\leq r$ and $\alpha_1\geq \beta_1,\ldots,\alpha_s\geq\beta_s$.

\subsection{IH-small resolutions of Bott-Samelson type}

Given a Schubert variety associated to $\underline{\alpha}$, we can extract two sequences $\bold{a}=(a_1,\ldots, a_d)$ and $\bold{q}=(q_1<\cdots<q_d)$ such that $a_i$ is the length of consecutive numbers in $\underline{\alpha}$, $q_i$ is the last number of the block from $\underline{\alpha}$. The equations for the Schubert variety $\mathbb{S}(\underline{\alpha})$ define the closure of 
$$\mathbb{S}(\underline{\alpha})^\circ=\{ L\;|\; \mathrm{dim}(L\cap V_{q_j})= a_1+\cdots+a_j\;\text{for}\; 1\leq i\leq d\}$$
associated to the $2\times d$ matrix of the form either
$$
\mathfrak{H}:=
\begin{bmatrix}
q_1 &\cdots& q_{d}\\
a_1 & \cdots& a_d\\
\end{bmatrix}_{\underline{\alpha}}\quad\text{or}\quad
\begin{bmatrix}
q_1 &\cdots& q_{d}&n\\
a_1 & \cdots& a_d&1\\
\end{bmatrix}_{\underline{\alpha}},
$$
based on its family. The following example illustrates the matrix. 
 
  \begin{example}
Let $n=7$ and $\underline{\alpha}=(2,3,5)$. The matrix for the Schubert variety $\mathbb{S}(\underline{\alpha})$ in  $OG'(n, V)$ is 
\[
\mathfrak{H}=
\begin{bmatrix}
3 &5\\
2 & 1\\
\end{bmatrix}_{\underline{\alpha}}
\]
of $d=2$.
\end{example}

We additionally have a sequence $\bold{b}=(b_0,\ldots, b_{d-1})$ from $\mathfrak{H}$ by setting
$$b_{i-1}=q_i-q_{i-1}-a_i$$
for $1\leq i \leq d$ and $b_d=n-q_d$ with $q_0=0$.
In order for a IH-small resolution for $\mathbb{S}(\underline{\alpha})$ to exist, there are two conditions imposed on $\underline{\alpha}$, $\bold{a},\bold{b}$ and $\bold{q}$ that for a sequence $\underline{\alpha}$, either $\alpha_s<n-s$ or $\alpha_s=n, \alpha_{s-1}\leq n-s$ holds for $s\geq 2$, and that $q_d<n-a_d$ and $q_d+(a_i+\cdots+a_d)<n+(b_i+\cdots +b_{d-1})$ are fulfilled for $i\geq 1$.
Under the suppositions, 
we can build the IH-small resolution of singularities for $\mathbb{S}(\underline{\alpha})$ inductively as follows. For notational convenience, we choose a Schubert variety in $OG'(n, V)$ but one can read this with $OG''(n,V)$.

The first step is to pick the smallest $i$ so that $b_i\leq a_i$ and $a_{i+1}\leq b_{i+1}$. (One may let $a_0=\infty$ and $b_{d}=\infty$.) We then take any subspace $U_1$ of $V$ of dimension $q_i+a_{i+1}$ such that $V_{q_i}\subseteq U_1\subseteq V_{q_{i+1}}$. For a fixed partial flag $0\subseteq V_{q_1}\subseteq \cdots \subseteq V_{q_{i-1}}\subseteq U_1\subseteq V_{q_{i+2}}\subseteq \cdots \subseteq V_{q_d}$, the Schubert variety $\mathbb{S}(\underline{\alpha^1})$ is defined
by the closure of
$$\mathbb{S}(\underline{\alpha^1})^\circ=\{ L\;|\; \mathrm{dim}(L\cap V_{q_j})= a_1+\cdots+a_j\;\text{for}\; j\neq i, i+1, \mathrm{dim}(L\cap U_1)= a_1+\cdots a_{i+1}\}.$$
Let us consider the locus
$$Z_1=\{(U_1,U)\;|\; \overline{U_1}\in Gr(a_{i+1}, V_{q_{i+1}}/V_{q_i}), V_{q_i}\subseteq U_1, U\in \mathbb{S}(\underline{\alpha^1}) \}\subseteq OG(q_i+a_{i+1}, V)\times OG'(n, V)
$$
with the second projection $p:Gr(a_{i+1}, V_{q_{i+1}}/V_{q_i})\times OG'(n,V)\rightarrow OG'(n,V)$.
The restriction of the projection $\pi_0=p|_{Z_1}:Z_1\rightarrow \mathbb{S}(\underline{\alpha})$ on $Z_1$ is a surjective birational morphism. Given the matrix
$$
\mathfrak{H}_1:=\begin{bmatrix}
q_1& \cdots& q_i+a_{i+1}&q_{i+2} &\cdots \\
a_1&\cdots&a_i+ a_{i+1}&a_{i+2}&\cdots
\end{bmatrix}_{\underline{\alpha^1}}
$$ associated to the variety $\mathbb{S}(\underline{\alpha^1})$,
we iterate this process to get the desingularization $\pi:Z_d\rightarrow \mathbb{S}(\underline{\alpha})$ as the composition of the morphisms $\pi_i:Z_{i+1}\rightarrow Z_i$. Here $Z_d$ is a subvariety of a product of $G/Q_i$ for a certain maximal parabolic subgroups $Q_i$, i.e., 
\begin{align*}
Z_d=&\{(U_{d},U_{d-1}, \cdots, U_1, U)\;|\; \overline{U_j}\in Gr(a_{j+1}, W^{R}_{{j}}/W^{L}_{j}), W^{L}_{j}\subseteq U_j, U\in \mathbb{S}(\underline{\alpha}^d) \}
\end{align*}
in $G/Q_1\times \cdots \times G/Q_{d}\times OG'(n,V)$ for each $j\in \{1,\ldots, d\}$.

We notice that the variety $Z_d$ relies on the incidence condition $W^{L}_{j}\subseteq U_j\subseteq W^{R}_{{j}}$ at each procedure and the last incident condition for $U=U_{d+1}$ becomes $W_{d+1}^L\subset U_{d+1}\subset V$. The following example gives the manner of finding the IH-small resolution.

\begin{example}
Let $G=SO(28)$ and $V$ be a vector space of dimension $28$ over $\mathbb{C}$. Let $0\subseteq V_1\subseteq \cdots \subseteq  V_{14}$ of $V$ denote a fixed (isotropic) partial flag whose subscript indicating its dimension, $\mathrm{dim}(V_{k})=k$. We select $i=0,1,2$ in this order to have a IH-small resolution for $\mathbb{S}(\underline{\alpha})$ associated to the matrix $\mathfrak{H}=
\begin{bmatrix}
3 &6&8\\
2 & 1&1\\
\end{bmatrix}_{\underline{\alpha}}$ . To begin with, we obtain the following variety
$$Z_1=\{(U_1,U)\:|\; 0\subset U_1\subset V_3, \; U\in\mathbb{S}(\underline{\alpha}^1)\}$$
in ${OG}(2, V)\times OG'(14,V)$ where 
$\mathbb{S}(\underline{\alpha}^1)$ is the closure of the locus defined by $\mathrm{dim}(L\cap V_6)= 3$ and $\mathrm{dim}(L\cap V_8)= 4$
associated to a new matrix
$\mathfrak{H}_1=
\begin{bmatrix}
2 &6&8\\
2 & 1&1\\
\end{bmatrix}_{\underline{\alpha}^1}.$ The next stage brings us to the variety
$$Z_2=\{(U_1,U_2,U)\;|\; 0\subset U_1\subset V_3,\; U_1\subset U_2\subset V_6, \;U\in \mathbb{S}(\underline{\alpha}^2) \}$$ in $OG(2,V)\times OG(3,V)\times OG'(14,V)$ with the variety $\mathbb{S}(\underline{\alpha}^2)$ associated to  
$\mathfrak{H}_2=
\begin{bmatrix}
3 &8\\
3 &1\\
\end{bmatrix}_{\underline{\alpha}^2}.$ Finally we acquire the resolution $Z=Z_2$ as 
$$Z=\{(U_1,U_2,U_3,U)\;|\;0\subset U_1\subset V_3,\; U_1\subset U_2\subset V_6, \; U_2\subset U_3\subset V_8, \;U\in \mathbb{S}(\underline{\alpha}^3)\}$$
for $\mathbb{S}(\underline{\alpha})$ inside $OG(2,V)\times OG(3,V)\times OG(4,V)\times OG'(14,V)$. Here $\mathbb{S}(\underline{\alpha}^3)$ is associated to $\mathfrak{H}_3=
\begin{bmatrix}
4 \\
4 \\
\end{bmatrix}_{\underline{\alpha}^3}.$ As $U\in \mathbb{S}(\underline{\alpha}^3)$ implies the closure of the locus of $\mathrm{dim}(U\cap U_3)=4$, we may replace the condition by $U_3\subset U\subset V$. 
Thereupon, the IH-small resolution for the Schubert variety $\mathbb{S}(\underline{\alpha})$ becomes the locus
$$Z_d=\{ (U_1,U_2,U_3,U)\:|\; 0\subset U_1\subset V_3,\; U_1\subset U_2\subset V_6, \; U_2\subset U_3\subset V_8, \; U_3\subset U\subset V\},$$
with the projection $Z\rightarrow \mathbb{S}(\underline{\alpha})$ sending $(U_1,U_2,U_3,U)$ to $U$.

\end{example}
From now on we write $Z_{\underline{\alpha}}$ in lieu of $Z_d$. To sum up, the following theorem is the overall aftermaths pertaining to the IH-small resolution for Schubert varieties.

\begin{theorem}[Sankaran and Vanchinathan]
Let $\mathbb{S}(\underline{\alpha})\subset OG'(n,V)$ (resp. $OG''(n,V)$) be a Schubert variety associated to a strictly increasing positive sequence $\underline{\alpha}$ of length $s$ where $n-s$ is even. Let $\mathfrak{H}$ be either
$$\begin{bmatrix}
q_1 &\cdots& q_{d}\\
a_1 & \cdots& a_d\\
\end{bmatrix}_{\underline{\alpha}}\quad\text{or}\quad
\begin{bmatrix}
q_1 &\cdots& q_{d}&n\\
a_1 & \cdots& a_d&1\\
\end{bmatrix}_{\underline{\alpha}}.$$ 
Let either $\alpha_s<n-s$ or $\alpha_s=n, \alpha_{s-1}\leq n-s$ for $s\geq 2$. Let $q_d<n-a_d$ and $(a_d+\cdots+a_i)-(b_{d-1}+\cdots+b_i)<n-q_d$ for $i\geq 1$. 
Then
\begin{enumerate}
\item $Z_{\underline{\alpha}}$ is a nonsingular projective variety.
\item The projection $\pi:Z_{\underline{\alpha}}\rightarrow OG'(n,V)$ is proper whose image is $\mathbb{S}(\underline{\alpha})$ and isomorphic over $\mathbb{S}(\underline{\alpha})^\circ$, so that it is a resolution of singularities.
\item  $\pi:Z_{\underline{\alpha}}\rightarrow \mathbb{S}(\underline{\alpha})$ is the IH-small resolution.  
\end{enumerate}
\end{theorem}

The sequences $\bold{a}$ and $\bold{b}$ from $\underline{\alpha}$ can be represented by a piecewise function $y=|x|$ in the $xy$-plane whose ascending and descending segments are $b_0,\ldots,b_{d-1}$ and $a_1,\ldots,a_d$ respectively. It has known that if we have $\mathbb{S}(\underline{\beta})\subset \mathbb{S}(\underline{\alpha})$, then the piecewise graph $y=\beta(x)$ for $\underline{\beta}$ lies below the one $y=\alpha(x)$ for $\underline{\alpha}$. The graph of these functions is depicted as Figure \ref{fig:image4}
with $\bold{b}=(3,1,2)$ and $\bold{a}=(3,2,1)$ for $\underline{\alpha}$ and $\bold{b}=(0,2,2)$ and $\bold{a}=(3,4,1)$ for $\underline{\beta}$. 
\begin{figure*}[h]

\begin{tikzpicture}[scale=0.31, x=40pt, y=30pt]
\draw[->] (-3 ,-3) -- (10.5, -3) node[right] {$x$};
     \path[draw, blue] (-2, 4) -- (1, 7) node[black, pos=0.7, left=4] {$b_0$}  ;
   \draw [fill] (-2, 4) circle [radius=0.1] node [above] {};  \draw [fill] (-1, 5) circle [radius=0.1] node [above] {};
\draw [fill] (0, 6) circle [radius=0.1] node [above] {};
\draw [fill] (1, 7) circle [radius=0.1] node [above] {};
\path[draw, blue] (1,7) -- (3, 5)  node[black, pos=0.3, right=6 , rotate=-30] {$a_1$} ;
\path[draw, blue] (3,5) -- (4, 4)  ;
\path[draw, thick, purple, dotted] (4,4) -- (4,2)   ;
\draw [fill] (2, 6) circle [radius=0.1] node [above] {};
\draw [fill] (3, 5) circle [radius=0.1] node [above] {};
\draw [fill] (4, 4) circle [radius=0.1] node [above] {};
\path[draw,blue] (4,4) -- (5, 5) node[black, pos=1.2, above left=0, rotate= 50] {$b_1$} ;
\draw [fill] (5, 5) circle [radius=0.1] node [above] {};
\path[draw, blue] (5,5) -- (6,4)   ;
\path[draw, blue] (6,4) -- (7,3)  node[black, pos=-0.8,  right=7,  rotate= -30] {$a_2$};
\path[draw, thick, purple, dotted] (7,3) -- (7,-1)   ;
\draw [fill] (6, 4) circle [radius=0.1] node [above] {};
\draw [fill] (7, 3) circle [radius=0.1] node [above] {};
\path[draw, blue] (7,3) -- (9, 5) node[black, pos=0.8, above left=0, rotate= 45] {$b_2$}  ;
\draw [fill] (8, 4) circle [radius=0.1] node [above] {};
\draw [fill] (9, 5) circle [radius=0.1] node [above] {};
\path[draw, blue] (9,5) -- (10,4) node[black, pos=-0.2,  right=4,  rotate= -20] {$a_3$} ;
\path[draw, thick, purple, dotted] (10,4) -- (10,0)   ;
\draw [fill] (10, 4) circle [radius=0.1] node [above, right=4] {$y=\alpha(x)$};
 \path[draw, red] (-2,4) -- (1, 1)  ;
\draw [fill] (-1, 3) circle [radius=0.1] node [above] {};
\draw [fill] (0, 2) circle [radius=0.1] node [above] {};
\draw [fill] (1, 1) circle [radius=0.1] node [above] {};
\draw [fill] (2, 2) circle [radius=0.1] node [above] {};
\path[draw, red] (1,1) -- (3, 3)  ;
\draw [fill] (3, 3) circle [radius=0.1] node [above] {};
\path[draw, red] (3,3) -- (7, -1)  ;
\draw [fill] (4, 2) circle [radius=0.1] node[black, below] {\textcolor{purple}{$c_1$}} ;
\draw [fill] (5, 1) circle [radius=0.1] node [above] {};
\draw [fill] (6, 0) circle [radius=0.1] node [above] {};
\draw [fill] (7, -1) circle [radius=0.1] node[black, below] {\textcolor{purple}{$c_2$}};
\path[draw, red] (7,-1) -- (9, 1)  ;
\draw [fill] (8, 0) circle [radius=0.1] node [above] {};
\draw [fill] (9, 1) circle [radius=0.1] node [above] {};
\draw [fill] (10, 0) circle [radius=0.1]  node[black, below] {\textcolor{purple}{$c_3$}} ;
 \path[draw, red] (9, 1) -- (10, 0) ;
 \draw [fill] (10, 0) circle [radius=0.1] node [above, right=4] {$y=\beta(x)$};
\end{tikzpicture}
  \caption{$(\mathfrak{H},\mathfrak{K})$-sequence}
    \label{fig:image4}
    \vspace*{-2mm}

\end{figure*}
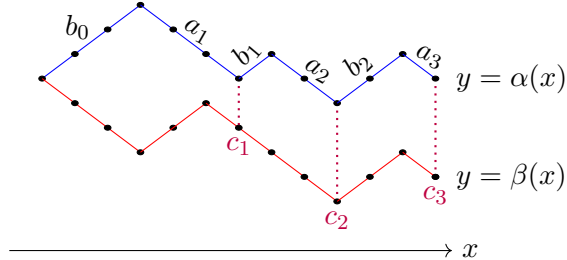

A {\it capacity} is a sequence $\bold{c}=(c_0,\ldots,c_d)$ of integers encoding $1/\sqrt{2}$ of the distance measured vertically from the local minimums to the graph of $y=\beta(x)$. It is advantageous to set $c_0=0$. In Figure \ref{fig:image4}, the capacity is $\bold{c}=(0,2,4,4).$

Let $\pi:Z_{\underline{\alpha}}\rightarrow \mathbb{S}(\underline{\alpha})$ be the IH-small resolution. In the event that $U$ is a point inside of the Schubert cell $\mathbb{S}(\underline{\beta})^\circ$, the Euler characteristic $d_{\mathfrak{H},\bold{c}}:=\chi(\pi^{-1}(U))$ of the fiber over a point $U\in \mathbb{S}(\underline{\beta})^\circ$ is obtained by the following formula.
%
\begin{theorem}[Sankaran and Vanchinathan]
Let $\mathfrak{H}:=
\begin{bmatrix}
q_1 &\cdots& q_{d}\\
a_1 & \cdots& a_d\\
\end{bmatrix}_{\underline{\alpha}}\;\text{or}\;
\begin{bmatrix}
q_1 &\cdots& q_{d}&n\\
a_1 & \cdots& a_d&1\\
\end{bmatrix}_{\underline{\alpha}}
$ with a sequence $\bold{b}=(b_0,\ldots, b_{d-1})$ and capacity $\bold{c}=(c_0,\ldots, c_d)$. Suppose that $i$ is the smallest integer such that $b_i\leq a_i$ and $a_{i+1}\leq b_{i+1}$ with $\mathfrak{H}_1$ as above. Then we have 
\begin{equation}\label{eqn:KL(1)}
d_{\mathfrak{H},\bold{c}}=\sum_{t\geq 0}\binom{a_{i+1}-c_i+c_{i+1}}{c_{i+1}-t}\binom{b_i+c_i-c_{i+1}}{c_i-t}d_{\mathfrak{H}_1,\bold{c}(t)}
\end{equation}
where $\bold{c}(t)=(c_0,\ldots, c_{i-1}, t, c_{i+2}, \ldots, c_d)$.
\end{theorem}

We remark that the IH-small resolution $\pi$ enables the function $d_{\mathfrak{H},\bold{c}}$ on the Schubert variety $\mathbb{S}(\underline{\alpha})$ to agree on the local Euler obstruction. The following proposition is useful for the computation of the Chern-Mather classes of Schubert varieties in the even orthogonal Grassmannians. 

\begin{proposition}\label{prop:equiv}
Let $\pi:Z_{\underline{\alpha}}\rightarrow \mathbb{S}(\underline{\alpha})$ be a IH-small resolution of a Schubert variety $\mathbb{S}(\underline{\alpha})$ in $OG'(n,V)$ (resp. $OG''(n,V)$). The following statements are equivalent.
\begin{enumerate}
        \item A point $U\in \mathbb{S}(\underline{\beta})^\circ\subset \mathbb{S}(\underline{\alpha})$ is smooth in $\mathbb{S}(\underline{\alpha})$
        \item $\pi^{-1}(U)$ is a point \label{a}
        \item $d_{\mathfrak{H},\bold{c}}=1$
          \item The capacity $\bold{c}$ is the sequence $(0,\ldots, 0)$
    \end{enumerate}
\end{proposition}
\begin{proof}
We prove the direction of $(4)$ to $(3)$, as the rest can be found in \cite[Proposition 4.2.6]{Benjamin}. Suppose $\bold{c}=(0,\ldots, 0)$. It follows from the construction that $t$ must be $0$. In this way $d_{\mathfrak{H},\bold{c}}=d_{\mathfrak{H}_1,\bold{c}(0)}$ via
$$\binom{a_j+c_j-c_{j-1}}{c_j-0}=\binom{b_{j-1}-c_j+c_{j-1}}{c_{j-1}-0}=1.$$
By induction, $d_{\mathfrak{H}_1,\bold{c}(0)}=1$ is deduced by $(c_0,\ldots,c_{j-1},0,c_{j+1},\ldots, c_d)=(0,\ldots, 0)$. Henceforth, the result follows.
\end{proof}

\section{Chern-Mather classes for Schubert varieties in the orthogonal Grassmannian $OG(n,\mathbb{C}^{2n})$}\label{sec4}
In this section we formulate an integration and its computation for the Schubert expansion in Chern-Mather class of Schubert varieties for type D in analogy to the version of type A by Jones \cite{Benjamin}. Our type D Chern-Mather class formula involves wedge products and Pfaffians (or Schur P-functions) that are a major different part from type A.

\subsection{Total Chern class of the IH-small resolutions}\label{s4.1}

Recall $V=\mathbb{C}^{2n}$ and the IH-small resolution $Z_{\underline{\alpha}}=\{(U_1,U_2,\ldots, U_d, U_{d+1})\;|\; W_i^L\subset U_i\subset W_i^R, W_{d+1}^L\subset U_{d+1}\subset V\}\subset X:=\prod_{j=1}^{d+1} OG(k_j,V)$ for $1\leq i\leq d$ for a Schubert variety $\mathbb{S}(\underline{\alpha})$ in the even orthogonal Grassmannian $OG'(n,V)$ (or $OG''(n,V)$), where $k_{d+1}=n$ with the projection map $pr_i:X\rightarrow OG(k_i,V)$. Let $\underline{V}_i$ be the isotropic subbundle of rank $i$ on $X$ whose fiber is $V_i$ from $V_\bullet$ and $\underline{U}_i$ the universal isotropic subbundle of $\underline{V}$ on $OG(k_i,V)$. By abuse of notation the subbundle can be seen as the pullback under the projection $pr_i$ to $X$ and $Z_{\underline{\alpha}}$.

We define $\underline{W}_i^L$ and $\underline{W}_i^R$ to be the isotropic subbundles of $\underline{V}$ on $X$, with a fiber over a point $U_\bullet$ as $W_i^L$ and $W_i^R$ respectively. Since $Z_{\underline{\alpha}}$ has the incidence relations $W_i^L\subset U_i\subset W_i^R$ for $1\leq i\leq d$ and $W_{d+1}^L\subset U_{d+1}\subset V$, there must be corresponding incidence conditions for the isotropic bundles on $Z_{\underline{\alpha}}$ as $\underline{W}_i^L\subset\underline{U}_i\subset\underline{W}_i^R$ for $1\leq i\leq d$ and $\underline{W}_{d+1}^L\subset \underline{U}_{d+1}\subset \underline{V}$, in which $\underline{W}_i^L$ and $\underline{W}_i^R$ are either an isotropic subbundle $\underline{V}_i$ or a universal subbundle $\underline{U}_i$.

Let $X^{(j)}:=\prod_{i=1}^jOG(k_i,V)$ and denote by $\rho_j:X\rightarrow X^{(j)}$ the projection map assigning $(U_1,\ldots, U_{d+1})$ to $(U_1,\ldots, U_j)$ for $1\leq j\leq d+1$. We set $Z^{(j)}=\rho_j(Z_{\underline{\alpha}})$, having the natural projections $Z^{(j)}\rightarrow Z^{(l)}$ for $j>l$. Above all, each $Z^{(j)}$ can be viewed as a (ordinary or orthogonal) Grassmannian bundle on $Z^{(j-1)}$ as follows. 
\begin{proposition}\label{D:prop}
Let $l_i=k_i-\mathrm{dim}(W_i^L)$.
For $2\leq j\leq d,$ the natural projections $Z^{(j)}\rightarrow Z^{(j-1)}$ is a Grassmannian bundle with a fiber identified with $Gr(l_j,W_j^R/W_j^L)$. In case of $j=d+1$, the fiber of the map $Z^{(d+1)}\rightarrow Z^{(d)}$ can be considered as $OG(k_{d+1},(W_{d+1}^L)^\perp/W_{d+1}^L)$. Furthermore $Z^{(1)}$ is isomorphic to the Grassmannian $Gr(l_1,W_1^R/W_1^L)$. 
\end{proposition}
\begin{proof}
It is known from the construction that  
$$Z^{(j)}=\{(U_1,\ldots, U_j)\;|\; W_i^L\subset U_i\subset W_i^R\;\text{for}\;1\leq i\leq j\}\subset X^{(j)}.$$
 By the constraint $k_j+a_{j+1}< n$ for $2\leq j\leq d$ of the IH-small resolution, all $W_i^R$ are subspaces of the maximal isotropic subspace $V_n$ which is trivial as isomorphic to $\mathbb{C}^n$. As a result, the fiber 
 $$\{U_j\;|\; W_j^L\subset U_j\subset W_j^R\}\subset OG(k_j,V)$$
 of $Z^{(j)}\rightarrow Z^{(j-1)}$ over a point $(U_1,\ldots, U_{j-1})\in Z^{(j-1)}$ is the ordinary Grassmannian $Gr(l_j,W_j^R/W_j^L)$.
 
 When it comes to the projection map $Z^{(d+1)}\rightarrow Z^{(d)}$, we have the fiber as 
 \begin{equation}\label{typeD:perp}
 \{U_{d+1}\;|\; W_{d+1}^L\subset U_{d+1}\subset V\}.
 \end{equation}
Since $U_{d+1}$ is isotropic, we earn the inclusion $U_{d+1}\subset (W_{d+1}^L)^\perp$ automatically. Consequently \eqref{typeD:perp} must be the orthogonal Grassmannian $OG(l_{d+1}, (W_{d+1}^L)^\perp/ W_{d+1}^L)$.

 Knowing that $W_1^L$ and $W_1^R$ are also subspaces of $V_n$, the last part of the proposition is verified.
 \end{proof}

Let $\underline{E}\rightarrow Y$ be a rank $n$ vector bundle of a smooth variety $Y$. Let $\pi:Gr(k,\underline{E})\rightarrow Y$ denote the ordinary Grassmannian bundle of $k$-dimensional subspaces of the fibers of $\underline{E}$ over $Y$. Since all the fiber of $\pi$ is smooth, $Gr(k,\underline{E})$ is nonsingular. The cokernel of the imbedding of $\pi^{-1}TY$ in $TGr(k,\underline{E})$ is the relative tangent bundle $T_{Gr(k,\underline{E})/Y}$ over $Y$, producing
$$0\longrightarrow \pi^{-1}TY\longrightarrow TGr(k,\underline{E})\longrightarrow T_{Gr(k,\underline{E})/Y}\longrightarrow 0.$$
Let $\underline{S}$ be the subbundle of the pullback $\pi^{-1}(\underline{E})$ and let $\underline{Q}$ be the quotient bundle on $Gr(k,\underline{E})$. 
Then by \cite[B.5.8]{Fulton} the relative tangent bundle $T_{Gr(k,\underline{E})/Y}$ is canonically isomorphic to 
\begin{equation}\label{D:eq}
\mathrm{Hom}(\underline{S},\underline{Q})\cong \underline{S}^\vee\otimes \underline{Q}.
\end{equation}
This isomorphism was used by Jones \cite{Benjamin} for type A. When $Y$ is a point, the Grassmannian bundle becomes the classical Grassmannian \cite[Section 6]{Fulton92}. Principally, we can make a connection with a classical geometry about tangent spaces of Grassmannian $Gr(k,E)$ of $k$-planes in a vector space $E$ of dimension $n$ over $\mathbb{C}$: for a subspace $\Lambda$ in $Gr(k,E)$, the tangent space of $Gr(k,E)$ at $\Lambda$ is naturally isomorphic to $\mathrm{Hom}(\Lambda,E/\Lambda)=\Lambda^\vee\otimes E/\Lambda$ \cite[Example 16.1]{Harris}.

Likewise of the ordinary Grassmannian case, we consider a vector bundle $\underline{\mathscr{E}}\rightarrow Y$ of rank $2n$ over a smooth variety $Y$ where $\underline{\mathscr{E}}$ is equipped with the quadratic form $q$ on it. Let $OG(k,\underline{\mathscr{E}})$ be the orthogonal Grassmannian bundle of dimension $k$ subspaces in the fibers of $\underline{\mathscr{E}}$ over $Y$. Let $p:OG(k,\underline{\mathscr{E}})\rightarrow Y$ be a projection map from $OG(k,\underline{\mathscr{E}})$ on $Y$ and $\underline{\mathscr{S}}$ the rank $k$ isotropic subbundle of $p^{-1}(\underline{\mathscr{E}})$. Then we obtain the following general fact regarding the relative tangent bundle $T_{OG(k,\underline{\mathscr{E}})/Y}$.

\begin{lemma}\label{D:lem}
The relative tangent bundle $T_{OG(k,\underline{\mathscr{E}})/Y}$ fits into a split exact sequence
$$0\rightarrow \underline{\mathscr{S}}^\vee\otimes\underline{\mathscr{S}}^\perp/\underline{\mathscr{S}}\rightarrow T_{OG(k,\underline{\mathscr{E}})/Y}\rightarrow \mathrm{\wedge}^2\;\underline{\mathscr{S}}^\vee\rightarrow 0,$$
so that we have
$$T_{OG(k,\underline{\mathscr{E}})/Y}\cong ( \underline{\mathscr{S}}^\vee\otimes\underline{\mathscr{S}}^\perp/\underline{\mathscr{S}})\oplus {\wedge}^2\;\underline{\mathscr{S}}^\vee.$$
\end{lemma}

\begin{proof}
We provide a proof of what seems to be this folklore lemma, inspired by \cite[Lemma 3.1]{Park} which is originated from Harris \cite[Example 16.1]{Harris}.

Let $Gr(k,\underline{\mathscr{E}})\rightarrow Y$ be the (ordinary) Grassmannian bundle of dimension $k$ subspaces in the fibers of $\underline{\mathscr{E}}$ over $Y$. Then we have $T_{OG(k,\underline{\mathscr{E}})/Y}\hookrightarrow T_{Gr(k,\underline{\mathscr{E}})/Y}=\underline{\mathscr{S}}^\vee\otimes\underline{\mathscr{E}}/\underline{\mathscr{S}}$. We define a map 
$$\phi:\underline{\mathscr{S}}^\vee\otimes\underline{\mathscr{E}}/\underline{\mathscr{S}}\rightarrow \underline{\mathscr{S}}^\vee\otimes \underline{\mathscr{S}}^\vee$$
of vector bundles by $\phi=\iota\otimes \psi$ for the identity map $\iota:\underline{\mathscr{S}}^\vee\rightarrow\underline{\mathscr{S}}^\vee$ and the map
$$\psi:\underline{\mathscr{E}}/\underline{\mathscr{S}}\rightarrow \underline{\mathscr{E}}/\underline{\mathscr{S}}^\perp\cong\underline{\mathscr{S}}^\vee.$$
Since $T_{OG(k,\underline{\mathscr{E}})/Y}$ is the inverse image $\phi^{-1}(\wedge^2\;\underline{\mathscr{S}}^\vee)$ of the wedge square $\wedge^2\;\underline{\mathscr{S}}^\vee$ from the symmetric form associated to the quadratic form $q$ for $\underline{\mathscr{E}}$, we have a surjective restriction map 
$$\phi|_{T_{OG(k,\underline{\mathscr{E}})/Y}}:T_{OG(k,\underline{\mathscr{E}})/Y}\rightarrow \wedge^2\;\underline{\mathscr{S}}^\vee$$
of $\phi$ to $T_{OG(k,\underline{\mathscr{E}})/Y}$ for $\wedge^2\;\underline{\mathscr{S}}^\vee\subset \underline{\mathscr{S}}^\vee\otimes \underline{\mathscr{S}}^\vee$. As the kernel $\mathrm{ker}(\phi)=\underline{\mathscr{S}}^\vee\otimes\underline{\mathscr{S}}^\perp/\underline{\mathscr{S}}$ of $\phi$ is included in $T_{OG(k,\underline{\mathscr{E}})/Y}$ by its definition, putting all together proves the lemma.
\end{proof}

In the same manner of \cite[Section 6]{Fulton92} we have a classical version for the tangent spaces of orthogonal Grassmannian at a point.

\begin{corollary}[]
For $k\leq n$, let $OG(k, \mathscr{E})$ be a orthogonal Grassmannian of isotropic $k$-planes in a vector space $\mathscr{E}$ of dimension $2n$. Let $\Lambda\in OG(k,\mathscr{E})$ be a $k$-plane. Then we have a natural identification
$$T_\Lambda OG(k,\mathscr{E})=(\Lambda^\vee\otimes\Lambda^\perp/\Lambda)\oplus\wedge^2\Lambda^\vee.$$
\end{corollary}


The above corollary is an analogy as to the tangent space of isotropic Grassmannians shown in the proof of \cite[Lemma 3.1]{Park}. The following theorem exhibits the Chern class of the tangent bundle of the locus $Z_{\underline{\alpha}}$ with respect to universal bundles on $X$.

\begin{theorem}\label{CTZ}
Let $Z_{\underline{\alpha}}\rightarrow \mathbb{S}(\underline{\alpha})$ be the IH-small resolution for a Schubert variety $\mathbb{S}(\underline{\alpha})\subset OG'(n,V)$ (or resp. $OG''(n,V)$) associated to $\underline{\alpha}\in W^P$. Then the Chern class of $TZ_{\underline{\alpha}}$ is given by
$$c(TZ_{\underline{\alpha}})=\prod_{i=1}^{d}c((\underline{U}_i/\underline{W}_i^L)^\vee\otimes(\underline{W}_i^R/\underline{U}_i))c\left(\mathrm{\wedge}^2(\underline{U}_{d+1}/\underline{W}_{d+1}^L)^\vee\right).$$
\end{theorem}
\begin{proof}

We know from Proposition \ref{D:prop} that $Z^{(j)}$ is an ordinary Grassmannian bundle over $Z^{(j-1)}$ for $1\leq j\leq d$. Thence, the similar argument in the proof of \cite[Theorem 4.3.3]{Benjamin} works. Particularly we attain
$$c(TZ^{(d)})=\prod_{i=1}^dc((\underline{U}_i/\underline{W}_i^L)^\vee\otimes(\underline{W}_i^R/\underline{U}_i)).$$

Recall that $l_{d+1}=k_{d+1}-\mathrm{dim}(W_{d+1}^L)$ and $Z^{(d+1)}$ is an isotropic Grassmannian bundle over $Z^{(d)}$. These implicate
$$Z^{(d+1)}\cong OG(l_{d+1},(\underline{W}_{d+1}^L)^\perp/\underline{W}_{d+1}^R)$$ with the projection map $OG(l_{d+1},(\underline{W}_{d+1}^L)^\perp/\underline{W}_{d+1}^R)\rightarrow Z^{(d)}$ restricted by the map $\phi:X^{(d+1)}\rightarrow X^{(d)}$. 
Since $\underline{U}_{d+1}/\underline{W}_{d+1}^L$ is the universal subbundle of $OG(l_{d+1},(\underline{W}_{d+1}^L)^\perp/\underline{W}_{d+1}^R)$, by Proposition \ref{D:lem} we have a canonical isomorphism
\begin{align*}
T_{OG(l_{d+1},(\underline{W}_{d+1}^L)^\perp/\underline{W}_{d+1}^R)/Z^{(d)}}\cong& ((\underline{U}_{d+1}/\underline{W}_{d+1}^L)^\vee\otimes((\underline{U}_{d+1}/\underline{W}_{d+1}^L)^\perp/\underline{U}_{d+1}/\underline{W}_{d+1}^L)\\
&\oplus \wedge^2(\underline{U}_{d+1}/\underline{W}_{d+1}^L)^\vee.
\end{align*}
As $(\underline{U}_{d+1}/\underline{W}_{d+1}^L)^\perp/\underline{U}_{d+1}/\underline{W}_{d+1}^L$ is trivial, we arrive at
$$c(T_{OG(l_{d+1},(\underline{W}_{d+1}^L)^\perp/\underline{W}_{d+1}^R)/Z^{(d)}})\cong c(\wedge^2(\underline{U}_{d+1}/\underline{W}_{d+1}^L)^\vee)$$
as desired.
%
%
%
\end{proof}


As to Schubert varieties, Pragacz \citelist{\cite{P1}\cite{P2}} validated that the cohomology class for the varieties in orthogonal or symplectic Grassmannians can be decided by {\it Schur P or Q-functions} which are certain universal polynomials in Pfaffians. Let us look into these two families of polynomials. \label{page:Pfa}

The first family is about the Q-functions. Let $\lambda=(\lambda_1>\ldots>\lambda_N)$ be a strict partition whose length $\ell{(\lambda)}$ is $N$. For $k\neq l$, we set
\begin{equation}\label{schuQ}
\widetilde{Q}_{kl}(\underline{E}):=c_k(\underline{E})\cdot c_l(\underline{E})+2\sum_{j=1}^l(-1)^jc_{k+j}(\underline{E})\cdot c_{l-j}(\underline{E}),
\end{equation}
satisfying $\widetilde{Q}_{kk}(\underline{E})=0$ and $\widetilde{Q}_{kl}(\underline{E})=-\widetilde{Q}_{lk}(\underline{E})$ in the Chow group $A_*(Y)$ of $Y$. Specifically $\widetilde{Q}_k(\underline{E}):=\widetilde{Q}_{k0}(\underline{E})=c_k(\underline{E})$ for $k\geq 0$. Assume that $N$ is even. If not, we may put $\lambda_N=0$. Then we define 
$$\widetilde{Q}_\lambda(\underline{E}):=\mathrm{Pf}(\widetilde{Q}_{\lambda_i\lambda_j}(\underline{E}))_{1\leq i<j\leq N}$$
where $\mathrm{Pf}$ indicates the Pfaffian of the skew-symmetric matrix. The Pfaffians for $\widetilde{Q}_\lambda$ form a basis of the ring
\begin{equation}\label{eq:gamma}
\Gamma=\mathbb{Z}[\widetilde{Q}_1,\widetilde{Q}_2,\ldots]/(\widetilde{Q}_k^2+2\sum_{j=1}^k(-1)^j\widetilde{Q}_{k+j}\widetilde{Q}_{k-j}, k\geq 1)
\end{equation}
over $\mathbb{Z}$.
The second family is about the P-functions. In this family, we may assume that the Chern class of the vector bundle $\underline{E}$ is divisible by $2$. We define 
$$\widetilde{P}_\lambda(\underline{E}):=\dfrac{1}{2^{\ell(\lambda)}}\widetilde{Q}_\lambda(\underline{E}).$$ 
In particular, $\widetilde{P}_i(\underline{E})=c_i(\underline{E})/2$. We observe from the equation \eqref{schuQ} that
$$\widetilde{P}_{kl}(\underline{E})=\widetilde{P}_{k}(\underline{E})\cdot\widetilde{P}_l(\underline{E})+2\sum_{j=1}^{l-1}(-1)^j\widetilde{P}_{k+j}(\underline{E})\cdot \widetilde{P}_{l-j}(\underline{E})+(-1)^l\widetilde{P}_{k+l}(\underline{E}).$$

Here is the lemma addressing the relation of the class of Schubert variety, a Schubert class in the even orthogonal Grassmannians to the Pfaffians.

\begin{lemma}[{\cite[Theorem 2.1]{PR}}]\label{lem:class}
Let $V$ be a $2n$-dimensional vector space over $\mathbb{C}$ and it is equipped with a nondegenerate quadratic form. Then the Schubert class for some partition $\underline{\alpha}$ in the Chow group $A_*(OG'(n,V))$ (resp. $A_*(OG''(n,V))$) is 
$$\left[\mathbb{S}(\underline{\alpha})\right]=\widetilde{P}_{\underline{\alpha}}(\underline{U}^\vee),$$
 where $\underline{U}$ is the tautological subbundle on $OG'(n,V)$ (resp. $OG''(n,V)$). 
\end{lemma}

In addition we have the dual Schubert class $\left[\widetilde{\mathbb{S}}(\underline{\alpha})\right]$ given by $\widetilde{P}_{\rho(n-1)\backslash \underline{\alpha}}(\underline{E}^\vee)$ for the strict partition $\rho_{n-1}=(n-1,n-2,\ldots, 1)$ such that 
$$\int_{OG'(n,V)} \left[\mathbb{S}(\underline{\alpha})\right]\cdot\left[\widetilde{\mathbb{S}}(\underline{\alpha})\right]=1.$$
 Here $\rho(n-1)\backslash \underline{\alpha}$ is the complement partition of $\underline{\alpha}$ in $\rho(n-1)$.
Another relevant reference for these discussions can be \cite[Section 2]{PR}.

The following theorem suggests the way of finding the coefficients in the Schubert class of the pushforward $\pi_*c_{SM}(Z_{\underline{\alpha}})$.

\begin{theorem}\label{Cst_D}
Let $\underline{U}$ be the pullback of the universal tautological subbundle on $OG'(n,V)$ (resp. $OG''(n,V))$. The coefficient $\gamma_{\underline{\alpha}, \underline{\beta}}$ of the Schubert class $\left[\mathbb{S}(\underline{\beta})\right]$ in $\pi_*c_{SM}(Z_{\underline{\alpha}})$ is computed by 
$$\gamma_{\underline{\alpha}, \underline{\beta}}=\int_{Z_{\underline{\alpha}}}c(TZ_{\underline{\alpha}})\cdot \widetilde{P}_{\rho(n-1)\backslash\underline{\beta}}(\underline{U}^\vee)\cap\left[Z_{\underline{\alpha}}\right].$$
\end{theorem}

\begin{proof}
Let $\widetilde{\mathbb{S}}({\underline{\beta}})$ be the dual Schubert variety to $\mathbb{S}(\underline{\beta})$. Since two Schubert classes are dual each other under the pairing of Poincar\'{e} duality, we have $\int_{OG'(n,V)}\left[\mathbb{S}(\underline{\beta})\right]\cdot \left[\widetilde{\mathbb{S}}({\underline{\beta}})\right]=1$. Then the constant $\gamma_{\underline{\alpha},\underline{\beta}}$ can be expressed by
\begin{align*}
\gamma_{\underline{\alpha}, \underline{\beta}}&=\int_{OG'(n,V)}\pi_*c_{SM}(Z_{\underline{\alpha}})\cdot \left[\widetilde{\mathbb{S}}({\underline{\beta}})\right].
\end{align*}
As the locus $Z_{\underline{\alpha}}$ is nonsingular such that $c_{SM}(Z_{\underline{\alpha}})=c(TZ_{\underline{\alpha}})\cap \left[Z_{\underline{\alpha}}\right]$, the integral becomes 
$$\int_{OG'(n,V)}\pi_*(c(TZ_{\underline{\alpha}})\cap \left[Z_{\underline{\alpha}}\right])\cdot \left[\widetilde{\mathbb{S}}({\underline{\beta}})\right].$$
Combined with the class $\left[\widetilde{\mathbb{S}}({\underline{\beta}})\right]=\widetilde{P}_{\rho(n-1)\backslash\underline{\beta}}(\underline{U}^\vee)\in A_*(OG'(n,V))\;$ (resp. $A_*(OG''(n,V)))\;$
from Lemma \ref{lem:class} and the projection formula \cite[Proposition 2.5(c)]{Fulton}, we conclude
\begin{align*}
\gamma_{\underline{\alpha}, \underline{\beta}}&=\int_{OG'(n,V)}\pi_*\left((c(TZ_{\underline{\alpha}})\cap \left[Z_{\underline{\alpha}}\right])\cdot \pi^*\left[\widetilde{\mathbb{S}}({\underline{\beta}})\right]\right)=\int_{Z_{\underline{\alpha}}}c(TZ_{\underline{\alpha}})\cdot \widetilde{P}_{\rho(n-1)\backslash\underline{\beta}}(\underline{U}^\vee)\cap \left[Z_{\underline{\alpha}}\right],
\end{align*}
suppressing the pullback notation for vector bundles.
(cf. \cite[proof of Lemma 12.1]{Fulton}). 
\end{proof}

We will discuss some properties of $\gamma_{\underline{\alpha},\underline{\beta}}$ later in Remark \ref{rmk}. It is well-known that the characteristic cycle of $\mathcal{IC}$-sheaf is irreducible for all cominuscule Schubert varieties in types A, C and D if and only if their Dynkin diagram is simply laced \cite{BF}. Since the Dynkin diagram of type D is simply laced, the characteristic cycle associated to such Schubert varieties is irreducible:

\begin{theorem}[{\cite[Theorem 7.1A]{BF}}]\label{Thm:irred}
Let $\mathbb{S}(\underline{{\alpha}})\subset OG'(n, V)$ (resp. $OG''(n,V)$) be a Schubert variety and $\mathcal{IC}^\bullet_{\mathbb{S}(\underline{\alpha})}$ be the corresponding intersection cohomology sheaf. Then
$$CC(\mathcal{IC}^\bullet_{\mathbb{S}(\underline{\alpha})})=\left[\overline{T^*_{\mathbb{S}(\underline{\alpha})^\circ}OG'(n,V)}\right].$$
\end{theorem}

The irreducibility of $CC(\mathcal{IC}_{\mathbb{S}(\underline{\alpha})}^\bullet)$ and Theorem \ref{thm:CM} enable us to have the Chern-Mather class of $\mathbb{S}(\underline{\alpha})$ via the pushforward of $c(TZ_{\underline{\alpha}})$ for the IH-small resolution $Z_{\underline{\alpha}}\rightarrow \mathbb{S}(\underline{\alpha})$ of the Schubert variety $\mathbb{S}(\underline{\alpha})$.

%
\begin{remark}
The Chern-Mather classes of Schubert varieties in (types A and D) Grassmannians are always positive \cite[Corollary 10.5, Proposition 10.3]{MS}. In other word, $\gamma_{\underline{\alpha},\underline{\beta}}>0$.
\end{remark}

\subsection{Explicit computations}\label{ExmD}

We recall some basic formulas in \cite{Fulton} before our explicit example-computation on the Chern-Mather class of a Schubert variety. 

Let $X$ be an algebraic variety over $\mathbb{C}$. Suppose that $\underline{E}$ and $\underline{F}$ are vector bundles of rank $e$ and $f$ respectively over $X$. Let $\lambda=(\lambda_1,\ldots,\lambda_N)$ and $\mu=(\mu_1,\ldots,\mu_N)$ be nonnegative decreasing integer sequences of length $N$ with $\mu_i\leq\lambda_i$ for $1\leq i\leq N$, i.e., $\mu\subset\lambda$. We denote by $|a_{i,j}|_{1\leq i,j\leq N}$ the determinant of the matrix $(a_{i,j})_{1\leq i,j\leq N}$ and $\binom{c}{d}$ the binomials. The integer $\mathcal{D}^N_{\lambda\mu}$ is defined by the determinant
$$\mathcal{D}^N_{\lambda\mu}=\left| \binom{\lambda_i+N-i}{\mu_j+N-j}\right|_{1\leq i,j\leq N}.$$
Using these notations, the total Chern class of the tensor product of $\underline{E}^\vee$ and $\underline{F}$ can be written as the sum 
\begin{equation}\label{eqn:tensor}
c(\underline{E}^\vee\otimes \underline{F})=\sum \mathcal{D}^e_{\lambda\mu}s_{{\mu}}(E)s_{\widetilde{\lambda'}}(F)
\end{equation}
over $\mu\subset\lambda$ for the partition $0\leq\lambda_n\leq\cdots\leq \lambda_1\leq f$ bounded by the rank $f$.
Here $s_\nu(\underline{A})$ is the Schur determinant \cite[Section 14.5]{Fulton} for a partition $\nu$ at the Segre classes of $\underline{A}$ and $\widetilde{\lambda'}$ is the conjugate partition to the partition $\lambda'=(f-\lambda_n,\ldots,f-\lambda_1)$.
 
If the Chern roots of $\underline{E}$ are $\alpha_1,\ldots, \alpha_r$, then the exterior power of $\underline{E}$ is given by the product
\begin{equation}\label{eqn:ext}
c_r(\wedge^p \underline{E})=\prod_{i_1<\cdots<i_r}(1+(\alpha_{i_1}+\cdots+\alpha_{i_r})t).
\end{equation}

Let us state the algebraic version of Bott Residue formula \cite[Theorem 3]{EG}. Suppose $X$ is a smooth, compact (or complete) projective variety and has a $T$-action on it. Given $T$-equivariant vector bundles $\underline{E}_1$, $\underline{E}_2,\ldots, \underline{E}_n$ over $X$, we denote by $P(\underline{E})$ a polynomial in the Chern classes of the vector bundles $\underline{E}_1,\ldots,\underline{E}_n$. We also denote by $c_j^T(\underline{E}_i)$ the equivariant Chern classes of $\underline{E}_i$ for $i=1,\ldots, n$ and $P^T(\underline{E})$ the polynomial in the equivariant Chern classes of the $\underline{E}_i$ for $i=1,\ldots,n$ which specializes to the polynomial $P(\underline{E})$. Let $\pi_{X*}:A_*^TX\rightarrow R_T$ be the push-forward induced by the projection $\pi_X\rightarrow pt$. One can replace $X$ by any component $F$ of $X^T$. \label{page:Pfa1}  

\begin{theorem}[Bott Residue Formula]\label{Bott}
The integral of $P(\underline{E})$ over $X$ is the sum 
$$\int_X P(\underline{E})\cap [X]=\sum_{F\subset X^T}\pi_{F_*}\left( \dfrac{P^T(\underline{E}|_F)\cap [F]_T}{c_{d_f}^T(N_F X)}\right),$$
over the torus-fixed point set $X^T$ of $X$ where $N_FX$ is the normal bundle over $X$ at the connected components $F$ such that $d_F$ is the rank of $N_FX$ as well as the codimension of $F$ in $X$.

\end{theorem}
We note that $X^T$ is also smooth, so that it has normal bundle $N_FX$. 
We additionally need the lemma below to apply Theorem \ref{Bott}.
\begin{lemma}[{\cite[Lemma 5.1.4]{Benjamin}}]
Let $X$ be a variety with a trivial $T$-action. Suppose $E_{\chi}\rightarrow X$ is a $T$-equivariant vector bundle of rank $r$ over $X$ and the $T$-action in $E_{\chi}$ is given by the character $\chi$. Then the $T$-equivariant Chern class of the vector bundle $E_{\chi}$ is
$$c_i^T(E_{\chi})=\sum_{j\leq i}\binom{r-j}{i-j}c_j(E_{\chi})\chi^{i-j}.$$
Furthermore, if $X$ is a point, it becomes
$$c_i^T(E_{\chi})=\binom{r}{i}\chi^i\in R_T,$$
since the only term contributed in the summation is $j=0$. Here $R_T\cong \mathbb{Z}\left[t_1,\ldots,t_n\right]$ is the $T$-equivariant Chow ring of a point.
\end{lemma}

Now we are ready to provide an example demonstrating the computation for the Chern-Mather class of a Schubert variety when $n=5$. Note that the set $Z_{\underline{\alpha}}^T$ of $T$-fixed points is finite in general \cite[Lemma 5.1.3]{Benjamin}. 
\begin{example}\label{ex:D}
Let $V$ be a vector space of dimension $10$ over $\mathbb{C}$, equipped with a quadratic form on it, and have the ordered basis 
$$e_1<\cdots<e_5<e_{\bar{5}}<\cdots<e_{\bar{1}}.$$ 

We deal with a Schubert variety $\mathbb{S}(\underline{\alpha})\subset OG''(5,V)$ of dimension $8$ with $\underline{\alpha}=(3,5)$ and compute the constant $\gamma_{\underline{\alpha},\underline{\beta}}$ for $\underline{\alpha}=(3,5),\underline{\beta}=(3,4)$. For $V_i=\langle e_1,\ldots,e_i \rangle$ for $1\leq i\leq 5$, we fix a complete isotropic flag 
$$V_\bullet=(0\subset V_1\subset \cdots\subset V_5\subset V_4^\perp\subset\cdots \subset V_1^\perp\subset V).$$
With the associated matrix $\mathfrak{H}=
\begin{bmatrix}
3 &5\\
1 & 1\\
\end{bmatrix}_{\underline{\alpha}}$ of $d=1$, the construction for the IH-small resolution of singularity for $\mathbb{S}(\underline{\alpha})$ leads to the locus $Z_{\underline{\alpha}}=Z_1$  as 
$$Z_{\underline{\alpha}}=\{(U_1,U_2)\;|\; 0\subset U_1\subset V_3,\;\;\mathrm{dim}(U_1\cap U_2)\geq 1\}\subset {OG}(1,V)\times {OG}''(5,V).$$

We proceed with Theorem \ref{CTZ} to reach the Chern class of the tangent bundle over $Z_{\underline{\alpha}}$ as
\begin{align*}
c(TZ_{\underline{\alpha}})&\cong c(\underline{U}_1^\vee\otimes (\underline{V}_3/\underline{U}_1))\cdot c(\mathrm{\wedge}^2(\underline{U}_2/\underline{U}_1)^\vee).
\end{align*}
As for the class of Schubert variety $\mathbb{S}(\underline{\beta})$ represented by the function $\widetilde{P}_{(2,1)}(U_2^\vee)$ in $A_*(OG''(n,V))$, we derive the dual Schubert class for $\widetilde{\mathbb{S}}(\underline{\beta})$ by the Pfaffian
\begin{align*}
\left[\widetilde{\mathbb{S}}(\underline{\beta})\right]=\widetilde{P}_{(4,3)}(\underline{U}_2^\vee),
\end{align*}
associated to the partition $\rho(4)\backslash (2,1)=(4,3)$ \cite[Page 13]{PR}. Combining all together, the coefficient $\gamma_{\underline{\alpha},\underline{\beta}}$ associated to $\underline{\alpha}$ and $\underline{\beta}$ is computed by the integration
$$\gamma_{\underline{\alpha},\underline{\beta}}=\int_{Z_{\underline{\alpha}}} ( c_1(\underline{U}_1^\vee\otimes (\underline{V}_3/\underline{U}_1))+c_1(\mathrm{\wedge}^2(\underline{U}_2/\underline{U}_1)^\vee))\cdot c_{4,3}(\underline{U}_2^\vee)\cap\left[Z_{\underline{\alpha}}\right].$$

In order to use the the Bott Residue Formula, let us describe the $T$-fixed points $Z_{\underline{\alpha}}^T$ of $Z_{\underline{\alpha}}$. Given $i_1\in\{1,2,3\}$ and $i_2,\ldots, i_5\in\{1,\ldots,5,\bar{5},\ldots,\bar{1}\}\backslash\{i_1, \bar{i_1}\}$ such that $e_{i_2}<\cdots< e_{i_5}$ and the number of barred integers in $\{i_1,\ldots, i_5\}$ are even, we have $24$ torus-fixed points $p_\bullet=(\langle e_{i_1}\rangle ,\langle e_{i_1},\ldots,e_{i_5}\rangle)$. For instance, if $i_1=1$ is taken, there are $8$ fixed points:
\begin{align*}
p_1:=&(\langle e_{1} \rangle\subset \langle e_1, e_2,e_3,e_4, e_5\rangle),\;\; p_2:=(\langle e_{1} \rangle\subset \langle e_1, e_2, e_3, e_{\bar{5}},e_{\bar{4}}\rangle),\\
 p_3:=&(\langle e_{1} \rangle\subset \langle e_1,e_2,e_4, e_{\bar{5}}, e_{\bar{3}}\rangle),\;\; p_4:=(\langle e_{1} \rangle\subset \langle e_1, e_2,e_5,e_{\bar{4}}, e_{\bar{3}}\rangle),\\
p_5:=&(\langle e_1 \rangle\subset \langle e_1,e_3,e_4, e_{\bar{5}},e_{\bar{2}}\rangle),\;\;p_6:=(\langle e_{1} \rangle\subset \langle e_1, e_3,e_5, e_{\bar{4}},e_{\bar{2}}\rangle),\\
p_7:=&(\langle e_{1} \rangle\subset \langle e_1,e_4,e_5, e_{\bar{3}}, e_{\bar{2}}\rangle),\;\; p_8:=(\langle e_{1} \rangle\subset \langle e_1, e_{\bar{5}}, e_{\bar{4}},e_{\bar{3}},e_{\bar{2}}\rangle).
\end{align*}

Suppose the weights of the $\mathbb{C}^*$-action on $V$ are $\left[w_1,\ldots,w_5,-w_5,\ldots,-w_1\right]=\left[1,\ldots,5,-5,\ldots,-1\right].$ Without loss of generality, we fix a $T$-fixed point $p_\bullet=(\langle e_{i_1}\rangle,\langle e_{i_1}, \ldots, e_{i_5}\rangle)$.  According to \eqref{eqn:tensor} and \eqref{eqn:ext}, the $\mathbb{C}^*$-equivariant Chern classes of the bundles restricted to a point are presented by
$$
c_1^{\mathbb{C}^*}(\underline{U}_1^\vee\otimes (\underline{V}_3/\underline{U}_1))=(-3w_{i_1})t,\quad c_1^{\mathbb{C}^*}(\mathrm{\wedge}^2(\underline{U}_2/\underline{U}_1)^\vee))=(-3w_{i_2}-\cdots-3w_{i_5})t,$$
\begin{equation*}\label{P43}
\widetilde{P}_{(4,3)}^{\mathbb{C}^*}(\underline{U}_2^\vee)=\dfrac{1}{2^{2}}(c_4^{\mathbb{C}^*}(\underline{U}_2^\vee)\cdot c_3^{\mathbb{C}^*}(\underline{U}_2^\vee)-2c_5^{\mathbb{C}^*}(\underline{U}_2^\vee)c_2^{\mathbb{C}^*}(\underline{U}_2^\vee)+c_6^{\mathbb{C}^*}(\underline{U}_2^\vee)c_1^{\mathbb{C}^*}(\underline{U}_2^\vee)-c_7^{\mathbb{C}^*}(\underline{U}_2^\vee)),
\end{equation*}
given the total $\mathbb{C}^*$-equivariant Chern class of $\underline{U}_2^\vee$ as 
$$c^{\mathbb{C}^*}(\underline{U}_2^\vee)=(1-w_{i_1}t)\cdot(1-w_{i_2}t)\cdot(1-w_{i_3}t)\cdot(1-w_{i_4}t)\cdot(1-w_{i_5}t).$$ In addition, because of $c_{d_f}^T(N_FX)=c_{\mathrm{dim}\;(X)}^T(TX)$, the denominator of the formula would be 
$$c_8^{\mathbb{C}^*}(TZ_{\underline{\alpha}})=3w_{i_1}^2\cdot(w_{i_2}+w_{i_3})\cdot(w_{i_2}+w_{i_4})\cdot(w_{i_2}+w_{i_5})\cdot(w_{i_3}+w_{i_4})\cdot(w_{i_3}+w_{i_5})\cdot(w_{i_4}+w_{i_5})t^8.$$
To that end, the single term for this point $p_\bullet$ of the Bott Residue formula applied yields the rational number
\begin{equation*}\label{p1}
\dfrac{(-3w_{i_1}-\cdots-3w_{i_5})\cdot \left({1/t^7} \cdot\widetilde{P}_{(4,3)}^{\mathbb{C}^*}(\underline{U}_2^\vee)\right)}{3w_{i_1}^2\cdot(w_{i_2}+w_{i_3})\cdot(w_{i_2}+w_{i_4})\cdot(w_{i_2}+w_{i_5})\cdot(w_{i_3}+w_{i_4})\cdot(w_{i_3}+w_{i_5})\cdot(w_{i_4}+w_{i_5})}.
\end{equation*}

Summing up the rational numbers over all the $24$ $T$-fixed points with the weights of $\mathbb{C}^*$-action on $V$ finally results in the constant
$$\gamma_{\underline{\alpha},\underline{\beta}}=6.$$

\end{example}
\begin{remark}
Since $\widetilde{Q}_{kk}$ vanishes for all $k$, any symmetric polynomials in $w^2_1,\ldots, w^2_n$ must set to be $0$ in any computations for the number $\gamma_{\underline{\alpha},\underline{\beta}}$. See \cite[Proposition 4.2]{PR} for details.
\end{remark}

Likewise we can accomplish the Chern-Mather class of a Schubert variety $\mathbb{S}(\underline{\alpha})\subset OG'(5,V)$ associated to $\underline{\alpha}$ as a sum indexed by $\underline{\beta}\subseteq\underline{\alpha}$. 
%
%
 Let us list partitions labelled for convenience as the followings:
 \begin{align*}
 (3,5)=\alpha_0, \quad& (3,4)=\beta_0,    \quad(1,4)=\beta_3, \quad(2,3,4,5)=\gamma_0,  \quad(1,2,3,5)=\gamma_3,\\
 (2,5)=\alpha_1, \quad &(2,4)=\beta_1,   \quad(1,3)=\beta_4,  \quad(1,3,4,5)=\gamma_1, \quad(1,2,3,4)=\gamma_4 \\
 (1,5)=\alpha_2,  \quad& (2,3)=\beta_2,  \quad(1,2)=\beta_5, \quad(1,2,4,5)=\gamma_2.
 \end{align*}

 Schubert varieties in $OG''(5,\mathbb{C}^{10})$ that admit their IH-small resolutions are the one associated to the $5$ partitions 
 \begin{equation*}
 (3,5)=\alpha_0, (2,5)=\alpha_1, (1,5)=\alpha_2, (1,3)=\beta_4,(1,2)=\beta_5.
 \end{equation*}

 In Table \ref{Table}, the left most column indicates the indices for the Schubert varieties having the Sankaran and Vanchinathan's IH-small resolution and the top row is for all indices $\underline{\beta}$ which is less than equal to the corresponding index $\underline{\alpha}$ in the first column. Using these partitions, the coefficients $\gamma_{\underline{\alpha},\underline{\beta}}$ of the Chern-Mather classes $\left[c_M(\mathbb{S}(\underline{\alpha}))\right]|_{\underline{\beta}}$ are listed below so that the Mather class of $\mathbb{S}(\underline{\alpha})$ is calculated by the sum of each row that corresponds to the classes of sub-Schubert varieties contained in $\mathbb{S}(\underline{\alpha})$.

\begin{table}[h]
\caption{Chern-Mather classes of Schubert varieties in $OG''(5,\mathbb{C}^{10})$}
\centering
\def\arraystretch{1.5}
$\begin{array}{cccccccccccccccc}
\hline
 	    &  \alpha_0&\beta_0 &\alpha_1& \beta_1&\alpha_2&\beta_2&\beta_3&\beta_4&\gamma_0&\beta_5&\gamma_1&\gamma_2&\gamma_3&\gamma_4  \\ \hline
\alpha_0 & 1 & 6&6  &34 &17&60&88&174&72&144&204&204&84&24 \\ 
 \alpha_1& \cdot & \cdot& 1 & 6 &5&16&28&68&24&70&92&112&52&16\\
 \alpha_2&\cdot  & \cdot & \cdot &\cdot &1&\cdot&6&16&\cdot&24&24 &44&24&8\\
 \beta_4&  \cdot&\cdot & \cdot &\cdot &\cdot&\cdot&\cdot&1&\cdot &4&4&14&14&8\\ 
  \beta_5& \cdot &\cdot & \cdot & \cdot&\cdot&\cdot&\cdot&\cdot&\cdot&1&\cdot& 4&6&4\\ \hline
\end{array}$
\label{Table}
\end{table}


One may represent a Schubert variety $\mathbb{S}(\underline{\alpha})$ by a Young diagram that corresponds to a partition $\lambda$ (or Young diagrams) as its codimension or a cohomology class $\left[\mathbb{S}(\underline{\beta})\right]=\widetilde{P}_{\lambda}(U^\vee)$, which appears in some other literatures, for instance \cite{H,MS}.

We notice that there is no such a direct pushforward of the IH-small resolution of singularity $Z_{\underline{\alpha}}\subset X$ to the type D flag variety as in the ordinary (type A) cases \cite[Section 5.4]{Benjamin}, attributed to the limitation of the IH-small resolution by Sankaran and Vanchinathan.

\section{Kazhdan-Lusztig classes of Schubert varieties}\label{sec5}
The characteristic cycle of $\mathcal{IC}$-sheaves over Schubert varieties in the Lagrangian Grassmannians may not be irreducible in contrast to Schubert varieties in the even orthogonal Grassmannians. This prevents us from directly handling the Chern-Mather classes of Schubert varieties in Grassmannians of type C. Instead, we establish Kazhdan-Lusztig classes of Schubert varieties in the Lagrangian Grassmannians on account of Theorem \ref{thm:KL} in this section. We also discuss Kazhdan-Lusztig classes and the Mather classes of Schubert varieties in the odd orthogonal Grassmannians later this section.

\subsection{Type C}\label{s5.1}

Let $V$ be a vector space of dimension $2n$ over $\mathbb{C}$, equipped with a nondegenerate symplectic form. We take a strictly increasing sequence $\underline{\alpha}=(1\leq\alpha_1<\alpha_2<\cdots<\alpha_s\leq n)$ of nonnegative integers. Inside a isotropic partial flag $Fl^C(\underline{\alpha};V)$ of type C, let $$V_{\alpha_1}\subset\cdots\subset V_{\alpha_s}\subset V$$
be an isotropic partial flag such that $\mathrm{dim}(V_{\alpha_i})=\alpha_i$. 
The Schubert variety $\mathbb{S}(\underline{\alpha})$ is defined to be 
\begin{equation*}\label{typeD:sch}
\mathbb{S}(\underline{\alpha})=\{L\;|\;\mathrm{dim}(L\cap V_{\alpha_i})\geq i\;\text{for all}\;1\leq  i\leq s\}
\end{equation*}
of $\mathrm{dim}\;\mathbb{S}(\underline{\alpha})=\sum_{i=1}^s\alpha_i+(n+1)(n-s)-\dfrac{1}{2}n(n+1)$ in the Lagrangian Grassmannian $LG(n, V)=Sp(2n)/P$ parametrizing the maximal isotropic subspaces of $V$.
Analogously this locus is the closure of the Schubert cell $\mathbb{S}(\underline{\alpha})^\circ$ in which the equality holds, and a Schubert variety $\mathbb{S}(\underline{\beta})$ associated to $\underline{\beta}=(1\leq \beta_1<\beta_2<\cdots<\beta_r\leq n)$ is included in $\mathbb{S}(\underline{\alpha})$ if $s\leq r$ and $\alpha_1\geq \beta_1,\ldots,\alpha_s\geq\beta_s$.


The construction of a IH-small resolution is exactly akin to the one for type D, but $OG'(n,V)$ (resp. $OG''(n,V)$) is replaced by $LG(n,V)$. That is, $Z_{\underline{\alpha}}\subset X^C:=\left(\prod_{i=1}^{d}LG(k_i,V)\right)\times LG(n,V)$ is the IH-small resolution of singularity for $\mathbb{S}(\underline{\alpha})$ where $W_j^L\subset U_j\subset W_j^R$ for $1\leq j\leq d$ and $W_{d+1}^L\subset U_{d+1}\subset V$. We summarize the facts concerning the IH-small resolution $Z_{\underline{\alpha}}\rightarrow \mathbb{S}(\underline{\alpha})$ for Schubert varieties in $LG(n,V)$ by Sankaran and Vanchinathan with assumptions for the IH-small resolution to exist.

\begin{theorem}[Sankaran and Vanchinathan]\label{thm:small}
Let $\mathbb{S}(\underline{\alpha})\subset LG(n,V)$ be a Schubert variety associated to $\underline{\alpha}=(1\leq\alpha_1<\alpha_2<\cdots<\alpha_s\leq n)$ and $\mathfrak{H}=\begin{bmatrix}
q_1 &\cdots& q_{d}\\
a_1 & \cdots& a_d\\
\end{bmatrix}_{\underline{\alpha}}.$
Let $\alpha_s\leq n-s$, $q_d<n+1-a_d$ and $(a_d+\cdots+a_i)-(b_{d-1}+\cdots+b_i)<n+1-q_d$ for $i\geq 1$.
Then
\begin{enumerate}
\item For such $\underline{\alpha}$, the locus $Z_{\underline{\alpha}}$ is a nonsingular projective variety.
\item $\pi:Z_{\underline{\alpha}}\rightarrow LG(n,V)$ is a proper mapping onto $\mathbb{S}(\underline{\alpha})$ and is isomorphic over $\mathbb{S}(\underline{\alpha})^\circ$. Thus it is a resolution of singularities.
\item $\pi:Z_{\underline{\alpha}}\rightarrow \mathbb{S}(\underline{\alpha})$ is the IH-small resolution. 
\end{enumerate}
\end{theorem}

The notations, which are not specified are adapted from Section \ref{sec4} and will be used for the rest of this article.

The proof of \cite[Proposition 4.2.6]{Benjamin} used that a point $U$ is smooth if and only if it is rationally smooth for Schubert varieties in type ADE \cite{CK}. Rationally smoothness approximates the smoothness via cohomological criteria and is related to the stalk Euler characteristic of the intersection cohomology sheaf. The equivalences in Proposition \ref{prop:equiv} are valid even for Schubert varieties in Lagrangian Grassmannians (type C), since smoothness implies the rational smoothness as for Schubert varieties. Rationally smoothness of the point $U$ implies the statement \eqref{a} of Proposition \ref{prop:equiv} as shown in the proof of \cite[Proposition 4.2.6]{Benjamin}.
In general, smoothness and rational smoothness for Schubert varieties in type C are not equivalence. Especially a rationally smooth Schubert variety is smooth only when it corresponds to an element that is $1\bar{2}$-avoiding in the Weyl group $W^C_n$ of type C (or an element embedded to a $4231$-avoiding in the Weyl group $W^A_{2n}$ of type A) \cite[Addendum 13.3]{BL}.


%
%

Let $Y$ be a smooth variety equipped with an isotropic vector bundle $\underline{\mathscr{E}}\rightarrow Y$ with respect to the symplectic form. Then we take the isotropic Grassmannian bundle $IG(k,\underline{V})$ of dimension $k$ subspaces of the fibers of $\underline{\mathscr{E}}$ over $Y$ with a projection map $p:IG(k,\underline{\mathscr{E}})\rightarrow Y$. 
We note that if $k=n$, the isotropic Grassmannian $IG(k,\underline{\mathscr{E}})$ is called the Lagrangian Grassmannian bundle $LG(n,\underline{\mathscr{E}})$. The following lemma is a widely renowned fact about the decomposition of the tangent bundle over a smooth variety, which is applied to find a description for the Chern class of a tangent bundle of the resolution of singularity. 

\begin{lemma}\label{C:lem}
 The relative tangent bundle $T_{IG(k,\underline{\mathscr{E}})/Y}$ fits into a split exact sequence
$$0\rightarrow \underline{\mathscr{S}}^\vee\otimes\underline{\mathscr{S}}^\perp/\underline{\mathscr{S}}\rightarrow T_{IG(k,\underline{\mathscr{E}})/Y}\rightarrow \mathrm{Sym}^2\;\underline{\mathscr{S}}^\vee\rightarrow 0,$$
so that
$$T_{IG(k,\underline{\mathscr{E}})/Y}\cong ( \underline{\mathscr{S}}^\vee\otimes\underline{\mathscr{S}}^\perp/\underline{\mathscr{S}})\oplus \mathrm{Sym}^2\;\underline{\mathscr{S}}^\vee.$$

\end{lemma}
\begin{proof}
The lemma follows by the proof contained in \cite[Lemma 3.1]{Park} motivated by \cite{Harris}. We take an even dimensional isotropic vector bundle on $Y$ instead of the complex vector space, equipped with a nondegenerate symplectic form. (c.f. see Lemma \ref{D:lem} for details.)
\end{proof}

According to \cite[Section 6]{Fulton92} as before, we explicitly state the following corollary implicitly contained in the proof of \cite[Lemma 3.1]{Park}.
\begin{corollary}[{\cite{Park}}]
For $k\leq n$, let $IG(k,\mathscr{E})$ be an isotropic Grassmannian of $k$-planes in a vector space $\mathscr{E}$ of dimension $2n$. Let $\Lambda\in IG(k,\mathscr{E})$ be a point. Then the tangent space $T_\Lambda IG(k,\mathscr{E})$ of the isotropic Grassmannian at $\Lambda$ is
$$T_\Lambda IG(k,\mathscr{E})=(\Lambda^\vee\otimes\Lambda^\perp/\Lambda)\oplus \mathrm{Sym}^2\Lambda^\vee.$$
\end{corollary}

We recall $X^C:=\left(\prod_{i=1}^{d}LG(k_i,V)\right)\times LG(n,V)$. As in type D, let $X^{(j)}:=\prod_{i=1}^jIG(k_i,V)$ and $\rho_j:X^C\rightarrow X^{(j)}$ be the projection sending $(U_1,\ldots, U_{d+1})$ to $(U_1,\ldots, U_j)$ for $1\leq j\leq d+1$. Let $Z^{(j)}=\rho_j(Z_{\underline{\alpha}})$, with the projection $Z^{(j)}\rightarrow Z^{(l)}$ for $j>l$. We can similarly deduce that $Z^{(j)}$ is a (ordinary or isotropic) Grassmannian bundle on $Z^{(j-1)}$. 
\begin{theorem}\label{CTZ_C}
Let $Z_{\underline{\alpha}}$ be the IH-small resolution $Z_{\underline{\alpha}}$ of singularity for a Schubert variety $\mathbb{S}(\underline{\alpha})\subset LG(n,V)$ associated to $\underline{\alpha}\in W^P$. The Chern class of the tangent bundle $TZ_{\underline{\alpha}}$ is
$$c(TZ_{\underline{\alpha}})=\prod_{i=1}^{d}c((\underline{U}_i/\underline{W}_i^L)^\vee\otimes(\underline{W}_i^R/\underline{U}_i))c\left(\mathrm{Sym}^2(\underline{U}_{d+1}/\underline{W}_{d+1}^L)^\vee\right)$$
in terms of universal bundles over $X^C$.
\end{theorem}
\begin{proof}
To be precise $Z^{(j)}$ is a Grassmannian bundle over $Z^{(j-1)}$ for $1\leq j\leq d$ and $Z^{(d+1)}$ is an isotropic Grassmannian bundle with respect to the symplectic form over $Z^{(d)}$. We thus have 
$$c(TZ^{(d)})=\prod_{i=1}^dc((\underline{U}_i/\underline{W}_i^L)^\vee\otimes(\underline{W}_i^R/\underline{U}_i))$$
and the projection
$$Z^{(d+1)}\cong IG(l_{d+1},(\underline{W}_{d+1}^L)^\perp/\underline{W}_{d+1}^R)\rightarrow Z^{(d)}$$ via the restriction of $\phi:X^{(d+1)}\rightarrow X^{(d)}$. 
Since Lemma \ref{C:lem} gives rise to a canonical isomorphism
\begin{align*}
T_{IG(l_{d+1},(\underline{W}_{d+1}^L)^\perp/\underline{W}_{d+1}^R)/Z^{(d)}}\cong &((\underline{U}_{d+1}/\underline{W}_{d+1}^L)^\vee\otimes((\underline{U}_{d+1}/\underline{W}_{d+1}^L)^\perp/\underline{U}_{d+1}/\underline{W}_{d+1}^L)\\
&\oplus Sym^2(\underline{U}_{d+1}/\underline{W}_{d+1}^L)^\vee
\end{align*}
for the universal subbundle $\underline{U}_{d+1}/\underline{W}_{d+1}^L$ of $IG(l_{d+1},(\underline{W}_{d+1}^L)^\perp/\underline{W}_{d+1}^R)$, the Chern class of the relative tangent bundle would have to be
$$c(T_{IG(l_{d+1},(\underline{W}_{d+1}^L)^\perp/\underline{W}_{d+1}^R)/Z^{(d)}})\cong c(Sym^2(\underline{U}_{d+1}/\underline{W}_{d+1}^L)^\vee)$$
by the same argument in the proof of Theorem \ref{CTZ}.
\end{proof}

The Schubert class {\cite[Theorem 2.1]{PR}} for some partition $\underline{\alpha}$ in the Chow group $A_*(LG(n,V))$ is 
$$\left[\mathbb{S}(\underline{\alpha})\right]=\widetilde{Q}_{\underline{\alpha}}(U^\vee),$$
 where $\underline{U}$ is the tautological subbundle on $LG(n,V)$. These classes of Schubert varieties form a basis of the ring $\Gamma$ in \eqref{eq:gamma}.
We are now in the position to calculate integrals for coefficients in the Schubert class of the pushforward $\pi_*c_{SM}(Z_{\underline{\alpha}})$ due to Theorem \ref{Bott}.

\begin{theorem}\label{thm:cst_C}
Let $\underline{U}$ be the pullback of the universal tautological subbundle on $LG(n,V)$. The coefficient $\gamma_{\underline{\alpha}, \underline{\beta}}$ of the Schubert class $\left[\mathbb{S}(\underline{\beta})\right]$ in $\pi_*c_{SM}(Z_{\underline{\alpha}})$ is computed by 
$$\gamma_{\underline{\alpha}, \underline{\beta}}=\int_{Z_{\underline{\alpha}}}c(TZ_{\underline{\alpha}})\cdot \widetilde{Q}_{\rho(n)\backslash\underline{\beta}}(\underline{U}^\vee)\cap\left[Z_{\underline{\alpha}}\right].$$
\end{theorem}
\begin{proof}
The overall argument of the proof basically resembles to the one in Theorem \ref{Cst_D}.
We denote by $\widetilde{\mathbb{S}}({\underline{\beta}})$ the dual Schubert variety to $\mathbb{S}(\underline{\beta})$ so that  $\int_{LG(n,V)}\left[\mathbb{S}(\underline{\beta})\right]\cdot \left[\widetilde{\mathbb{S}}({\underline{\beta}})\right]=1$ is satisfied. By the duality, the constant $\gamma_{\underline{\alpha},\underline{\beta}}$ is given by the integration
\begin{align*}
\gamma_{\underline{\alpha}, \underline{\beta}}&=\int_{LG(n,V)}\pi_*c_{SM}(Z_{\underline{\alpha}})\cdot \left[\widetilde{\mathbb{S}}({\underline{\beta}})\right]\labelrel={eqn;sm a}\int_{LG(n,V)}\pi_*(c(TZ_{\underline{\alpha}})\cap \left[Z_{\underline{\alpha}}\right])\cdot \left[\widetilde{\mathbb{S}}({\underline{\beta}})\right],
\end{align*}
where $c_{SM}(Z_{\underline{\alpha}})=c(TZ_{\underline{\alpha}})\cap \left[Z_{\underline{\alpha}}\right]$ is applied for the equality ~\eqref{eqn;sm a}. 
The use of the fact $$\left[\widetilde{\mathbb{S}}({\underline{\beta}})\right]=\widetilde{Q}_{\rho(n)\backslash\underline{\beta}}(\underline{U}^\vee)\in A_*(LG(n,V))$$ and the projection formula establishes
\begin{align*}
\gamma_{\underline{\alpha}, \underline{\beta}}&=\int_{LG(n,V)}\pi_*\left((c(TZ_{\underline{\alpha}})\cap \left[Z_{\underline{\alpha}}\right])\cdot \pi^*\left[\widetilde{\mathbb{S}}({\underline{\beta}})\right]\right)=\int_{Z_{\underline{\alpha}}}c(TZ_{\underline{\alpha}})\cdot \widetilde{Q}_{\rho(n)\backslash\underline{\beta}}(\underline{U}^\vee)\cap \left[Z_{\underline{\alpha}}\right],
\end{align*}
suppressing the pullback notation for vector bundles. 
\end{proof}

In general the KL-classes can be written as a linear combination of Chern-Mather classes explicitly if we know the Euler obstruction corresponding to each pair of $\underline{\alpha}$ and $\underline{\beta}$. The formula for the (torus equivariant) Mather class of cominuscule Schubert varieties of type C (with the other types) is given by \cite{MS}.

The positivity of the constant $\gamma_{\underline{\alpha},\underline{\beta}}$ can be addressed as follows.
\begin{proposition}\label{prop:pos}
Let $V$ be a $2n$-dimensional vector space. Let $\mathbb{S}(\underline{\alpha})$ be a Schubert variety in Lagrangian Grassmannian $LG(n,V)$. In the Schubert expansion of the KL class
\[
KL(\mathbb{S}(\underline{\alpha}))=\sum_{\underline{\beta}}\gamma_{\underline{\alpha},\underline{\beta}}\left[\mathbb{S}(\underline{\beta})\right],
\]
the coefficient $\gamma_{\underline{\alpha},\underline{\beta}}$ is positive.
\end{proposition}
\begin{proof}
The proof of the statement is straightforward by the reasoning in \cite[Proposition 10.3]{MS}. That is, the Kazhdan-Lusztig polynomials associated to $\underline{\alpha}, \underline{\beta}$ with $\underline{\alpha}\geq \underline{\beta}$ are nonnegative and its constant term equals $1$. Hence we have the proposition by the equation \eqref{eq:KLspan} and the fact that CSM classes of Schubert cells in a homogeneous space $G/P$ are nonnegative \cite{H,AMSS} for a (complex) simple Lie group $G$  (in particular for $GL(n), Sp(n), SO(2n+1)$ and $SO(2n)$) and any parabolic subgroup $P\subset G$.
\end{proof}
\begin{remark}\label{rmk}
This proposition is independently proved by Aluffi, Mihalcea, Schuermann and Su \cite{AMSS} according to a private communication with one of the authors, Mihalcea. This statement will be included in their paper. In particular, the proof of Prop. \ref{prop:pos} has shown that $\gamma_{\underline{\alpha},\underline{\beta}}$ is positive regardless of types. In other words, the positivity property of $\gamma_{\underline{\alpha},\underline{\beta}}$ works for Schubert varieties in $G/P$ of any classical types. 

Here are another interesting observations related to the coefficients $\gamma_{\underline{\alpha},\underline{\beta}}$. In case of Mather classes, the Mather polynomial of $\mathbb{S}(\underline{\alpha})$ is a polynomial in $x^{\ell_{\underline{\beta}}}$ corresponding to $\left[\mathbb{S}(\underline{\beta})\right]$ where $\ell_{\underline{\beta}}$ is the dimension of $\mathbb{S}(\underline{\beta})$, and it is known by conjectures in \cite[Intro.]{MS} that the Mather polynomial is unimodal whose terminology is defined in \cite{Sta89}. In Example \ref{ex:D} for (type D) even orthogonal Grassmannian cases, the Mather polynomial $M_{\underline{\alpha}_0}(x)$ of $\mathbb{S}(\underline{\alpha}_0)$ is given by
\[
M_{\underline{\alpha}_0}(x)=x^8+12x^7+51x^6+148x^5+244x^4+348x^3+204x^2+84x+24
\]
which strengthens their unimodality conjectures in \cite[\S 10.4]{MS}.
\end{remark}

We shift gears to compute an example for the coefficient $\gamma_{\underline{\alpha},\underline{\beta}}$ of the Schubert expansion.

\begin{example}
Let $\underline{\alpha}=(2),\underline{\beta}=(1)$.
Let $V$ be a $6$-dimensional vector space over $\mathbb{C}$ with the ordered basis 
$$e_1<e_2<e_3<e_{\bar{3}}<e_{\bar{2}}<e_{\bar{1}}.$$ 
We fix a complete isotropic flag 
$$V_\bullet=(0\subset V_1\subset V_2\subset V_3\subset V_2^\perp\subset V_1^\perp\subset V)$$
where $V_i=\langle e_1,\ldots,e_i \rangle$ for $1\leq i\leq 3$ and consider the variety $\mathbb{S}{(\underline{\alpha})}=\{L\;|\;\mathrm{dim}(L\cap V_2)\geq 1\}\subset LG(3,V)$ of dimension $4$ with the IH-small resolution of singularity $Z_{\underline{\alpha}}\rightarrow \mathbb{S}{(\underline{\alpha})}$ by Sankaran and Vanchinathan. In light of the fixed partial flag $V_\bullet:0\subset V_2\subset V$, we have the locus $Z_{\underline{\alpha}}$ as
$$Z_{\underline{\alpha}}=\{(U_1,U_2)\;|\;0\subset U_1\subset V_2,\;\;\mathrm{dim}(U_2\cap U_1)\geq 1\}\subset IG(1,V)\times LG(3,V).$$ 
By virtue of Theorem \ref{thm:cst_C} we get
\begin{align*}
c(TZ_{\underline{\alpha}})&\cong c(\underline{U}_1^\vee\otimes \underline{V}_2/\underline{U}_1)\cdot c(\mathrm{Sym}^2(\underline{U}_2/\underline{U}_1)^\vee).
\end{align*}
Moreover we know that $\left[\mathbb{S}(\underline{\beta})\right]=\widetilde{Q}_{(3)}(U_2^\vee)$ so that the dual Schubert class of $\widetilde{\mathbb{S}}(\underline{\beta})$ is given by 
\begin{align*}
\left[\widetilde{\mathbb{S}}(\underline{\beta})\right]=\widetilde{Q}_{(2,1)}(U_2^\vee),
\end{align*}
as $\rho(3)\backslash (3)=(2,1)$. 
Then the constant $\gamma_{\underline{\alpha},\underline{\beta}}$ is obtained by the integration
$$\gamma_{\underline{\alpha},\underline{\beta}}=\int_{Z_{\underline{\alpha}}} ( c_1(\underline{U}_1^\vee\otimes \mathbb{C}^2/\underline{U}_1)+c_1(\mathrm{Sym}^2(\underline{U}_2/\underline{U}_1)^\vee))\cdot c_{2,1}(\underline{U}_2^\vee)\cap\left[Z_{\underline{\alpha}}\right].$$

We evaluate the integral, using the Bott Residue Formula with $8$ $T$-fixed points 
$$p_\bullet=(\langle e_{i_1}\rangle, \langle e_{i_1},e_{i_2},e_{i_3}\rangle)$$
of $Z^T$, by the choses of $i_1\in\{1,2\}$ and $i_2, i_3\in\{1,2,3,\bar{3},\bar{2},\bar{1}\}\backslash\{i_1, \bar{i_1}\}$ such that $e_{i_2}< e_{i_3}$.
%
%
%
Suppose that the weights of $\mathbb{C}^*$-action on $V$ are $\left[w_1,w_2,w_3,-w_3,-w_2,-w_1\right]=\left[1,2,3,-3,-2,-1\right].$ Let us fix a $T$-fixed point $p_\bullet=(\langle e_{i_1}\rangle,\langle e_{i_1}, e_{i_2}, e_{i_3}\rangle)$. Then the relevant $\mathbb{C}^*$-equivariant Chern classes of the bundles at a point are
$$
c_1^{\mathbb{C}^*}(\underline{U}_1^\vee\otimes \underline{V}_2/\underline{U}_1)=(-2w_{i_1})t,\quad c_1^{\mathbb{C}^*}(\mathrm{Sym}^2(\underline{U}_2/\underline{U}_1)^\vee))=(-3w_{i_2}-3w_{i_3})t,$$
$$c_{2,1}^{\mathbb{C}^*}(\underline{U}_2^\vee)=(-w_{i_1}^2w_{i_2}-w_{i_1}^2w_{i_3}-w_{i_1}w_{i_2}^2-w_{i_1}w_{i_2}w_{i_3}-w_{i_1}w_{i_3}^2-w_{i_2}^2w_{i_3}-w_{i_2}w_{i_3}^2)t^3,$$
$$c_4^{\mathbb{C}^*}(TZ_{\underline{\alpha}})=(8w_{i_1}w_{i_2}w_{i_3}(w_{i_2}+w_{i_3}))t^4.$$
Therefore the rational number as a term for the $T$-fixed point $p_\bullet$ of the Bott Residue Formula applied is 
\begin{equation*}\label{p1}
\dfrac{(-2w_{i_1}-3w_{i_2}-3w_{i_3})\cdot (-w_{i_1}^2w_{i_2}-w_{i_1}^2w_{i_3}-w_{i_1}w_{i_2}^2-w_{i_1}w_{i_2}w_{i_3}-w_{i_1}w_{i_3}^2-w_{i_2}^2w_{i_3}-w_{i_2}w_{i_3}^2)}{8w_{i_1}w_{i_2}w_{i_3}(w_{i_2}+w_{i_3})}.
\end{equation*}
We add all over these $8$ $T$-fixed points with the weights to have the value $$\gamma_{\underline{\alpha},\underline{\beta}}=5.$$

\end{example}

 The Kazhdan-Lusztig class of the Schubert variety $\mathbb{S}(\underline{\alpha})$ admitting the IH-small resolution for $n=3$ with respect to the (homology) class of Schubert varieties $\mathbb{S}(\underline{\beta})\subset \mathbb{S}(\underline{\alpha})$ is displayed in Table \ref{TableKL} whose left most column represents $\underline{\alpha}$ associated to $\mathbb{S}(\underline{\alpha})$ and the corresponding row indicates $\underline{\beta}$ such that the coefficients $\gamma_{\underline{\alpha},\underline{\beta}}$ of the Schubert class $[\mathbb{S}(\underline{\beta})]$ in the Kazhdan-Lusztig class $KL(\mathbb{S}(\underline{\alpha}))$ are listed. The Schubert expansion of the Kazhdan-Lusztig class for $\mathbb{S}(\underline{\alpha})$ is recovered by summing up the rows corresponding to the Schubert varieties $\mathbb{S}(\underline{\beta})\subset\mathbb{S}(\underline{\alpha})$. Here $(2)=\alpha_0,(2,3)=\beta_0,(1)=\alpha_1, (1,2,3)=\gamma_0, (1,3)=\beta_1,  (1,2)=\beta_2.$
\begin{table}[h]
\caption{}
\centering
\def\arraystretch{1.5}
$\begin{array}{ccccccccccccc}
\hline
 	    &  \alpha_0&\beta_0 &\alpha_1& \beta_1&\beta_2&\gamma_0  \\ \hline
\alpha_0 & 1 & 3&5  &14 &20&8 \\ 
\alpha_1 & \cdot & \cdot&1  & 3&8&4 \\ \hline
\end{array}$
\label{TableKL}
\end{table}


We notice that in $n=4$ the local Euler obstruction $\mathrm{Eu}_{\mathbb{S}(\underline{\alpha})}(p_{\underline{\beta}})$ of a Schubert variety $\mathbb{S}(\underline{\alpha})$ at $T$-fixed points $p_{\underline{\beta}}\in \mathbb{S}(\underline{\alpha})$ is exactly the same as the value of the Kazhdan-Lusztig polynomials $P_{\underline{\alpha},\underline{\beta}}(1)$ evaluated at $q=1$ by \cite[Table. 3]{MS} and \eqref{eqn:KL(1)}. In other words the Kazhdan-Lusztig class of  the Schubert variety $\mathbb{S}(\underline{\alpha})$ for $\underline{\alpha}=(1,2,4)$ (equivalently, the Young diagram ${\tiny \Tableau{
&&&\\
~&&&\\
~&~&\\
}}$ ) is equal to the Chern-Mather class of $\mathbb{S}(\underline{\alpha})$ presented in \cite[Example 6.4]{MS}, as

$$\begin{array}{cccccccccccccccc}
\hline
 	    &  \alpha_0&\beta_0 &\alpha_1& \beta_1&\alpha_2&\beta_2&\beta_3&\beta_4&\gamma_0&\beta_5&\gamma_1&\gamma_2&\gamma_3&\gamma_4  \\ \hline
\alpha_0 & 1 & 4&7  &27 &25&60&92&241&45&269&183&246&132&24 \\  \hline
\end{array}$$
where
\begin{equation}\label{A_Label}
 \begin{split}
 (3)=\alpha_0, \quad& (3,4)=\beta_0, \quad(1,4)=\beta_3, \quad(2,3,4)=\gamma_0,\quad(1,2,3)=\gamma_3,\\
 (2)=\alpha_1, \quad &(2,4)=\beta_1, \quad(1,3)=\beta_4,   \quad(1,3,4)=\gamma_1, \quad(1,2,3,4)=\gamma_4, \\
 (1)=\alpha_2,  \quad& (2,3)=\beta_2, \quad(1,2)=\beta_5, \quad(1,2,4)=\gamma_2.\\
 \end{split}
 \end{equation}
 %

\subsection{Type B}\label{s5.2}


We take a $2n+1$-dimensional vector space $V$ over $\mathbb{C}$ together with a non-degenerate quadratic form on it, and fix an isotropic partial flag 
$$V_{\alpha_1}\subset \cdots\subset V_{\alpha_s}\subseteq V$$
in $Fl^B(\underline{\alpha};V)$ where the rank of the subspaces is $\alpha_i$ of type C.
  We define the Schubert variety by
 $$\mathbb{S}(\underline{\alpha})=\{L\;|\;\mathrm{dim}(L\cap V_{\alpha_i})\geq i\;\text{for}\; 1\leq i\leq s\}$$
 in the odd orthogonal Grassmannian $OG(n,V)=SO(2n+1)/P$ of dimension $n$ isotropic subspaces of $V$.
 
 Along with the isomorphism $$\eta:OG(n,\mathbb{C}^{2n+1})\rightarrow OG'(n+1,\mathbb{C}^{2n+2})\;\text{ (resp.} \;OG''(n+1,\mathbb{C}^{2n+2}))$$
in \cite[Section 3.5]{IMN}, there is a Schubert variety $\mathbb{S}'(\underline{\alpha})$ such that the inverse image of $\mathbb{S}'(\underline{\alpha})$ under $\eta$ is the Schubert variety
$\eta^{-1}(\mathbb{S}'(\underline{\alpha}))=\mathbb{S}(\underline{\alpha})$, and the IH-small resolution $Z'_{\underline{\alpha}}$ for $\mathbb{S}'(\underline{\alpha})$ in the even orthogonal Grassmannian $OG'(n+1,\mathbb{C}^{2n+2})$ for some $\underline{\alpha}'$ of type D can be pulled back to the IH-small resolution $Z_{\underline{\alpha}}$ for the Schubert variety $\mathbb{S}(\underline{\alpha})$ in $OG(n,\mathbb{C}^{2n+1})$ by the diagram
\begin{equation}\label{comm}
\begin{tikzcd}[row sep=scriptsize, column sep=scriptsize]
Z_{\underline{\alpha}} \arrow[r, hook, crossing over]\arrow[dd,"\pi"]  & Z_{\underline{\alpha}}' \\
&  & &\\
\mathbb{S}(\underline{\alpha}) \arrow[r, hook,"\eta"]  & \mathbb{S}'(\underline{\alpha})\arrow[from=uu, crossing over,"pr_1"].
\end{tikzcd}
\end{equation}

Let $X^B:=\prod_{j=1}^{d+1} OG(k_j,V)$ for some $k_j, 1\leq j\leq d$ and $k_{d+1}=n$ so that $X^B$ contains the locus $Z_{\underline{\alpha}}$ defined by $W_j^L\subset U_j\subset W_j^R$, $\mathrm{dim}(U_j)=k_j$ for all $j$. Theorem \ref{CTZ_B} shows the Chern class of the tangent bundle $TZ_{\underline{\alpha}}$ as regards universal bundles on $X^B$.

\begin{theorem}\label{CTZ_B}
For a IH-small resolution $Z_{\underline{\alpha}}$ for a Schubert variety $\mathbb{S}(\underline{\alpha})$ associated to $\underline{\alpha}\in W^P$ in $OG(n,V)$, the Chern class of the tangent bundle $TZ_{\underline{\alpha}}$ on $Z_{\underline{\alpha}}$ is 
$$c(TZ_{\underline{\alpha}})=\prod_{i=1}^{d}c((\underline{U}_i/\underline{W}_i^L)^\vee\otimes(\underline{W}_i^R/\underline{U}_i))c((\underline{U}_{d+1}/\underline{W}_{d+1}^L)^\vee\otimes(\underline{U}_{d+1}^\perp/\underline{U}_{d+1}) )c\left(\mathrm{\wedge}^2(\underline{U}_{d+1}/\underline{W}_{d+1}^L)^\vee\right)$$
with respect to the universal bundles on $X^B$.
\end{theorem}
\begin{proof}
The proof is almost identical with Theorem \ref{CTZ}, which boils down to check the difference at the canonical isomorphism for the relative tangent bundle
\begin{align*}
T_{OG(l_{d+1},(\underline{W}_{d+1}^L)^\perp/\underline{W}_{d+1}^R)/Z^{(d)}}\cong& ((\underline{U}_{d+1}/\underline{W}_{d+1}^L)^\vee\otimes((\underline{U}_{d+1}/\underline{W}_{d+1}^L)^\perp/\underline{U}_{d+1}/\underline{W}_{d+1}^L)\\
&\oplus \wedge^2(\underline{U}_{d+1}/\underline{W}_{d+1}^L)^\vee.
\end{align*}
We know from $\mathrm{rk}(\underline{U}_{d+1})=n$ that $(\underline{U}_{d+1}/\underline{W}_{d+1}^L)^\perp/(\underline{U}_{d+1}/\underline{W}_{d+1}^L)\cong \underline{U}_{d+1}^\perp/\underline{U}_{d+1}$ is the line bundle that is equivalently isomorphic to $\wedge^{2n+1}\underline{V}$. It follows that the first equivariant Chern class $c_1^T(\underline{U}_{d+1}^\perp/\underline{U}_{d+1})$ of the line vanishes as $0$ \cite[Page. 75]{FP}. Thus, using the sequence vector bundles $Z^{(j)}\rightarrow Z^{(j-1)}$ where $Z^{(j)}=\{(U_1,\ldots, U_j)\;|\; W_i^L\subset U_i\subset W_i^R\;\text{for}\;1\leq i\leq j\}$, the Chern class of the $d$-th relative tangent bundle $T_{OG(l_{d+1},(\underline{W}_{d+1}^L)^\perp/\underline{W}_{d+1}^R)/Z^{(d)}}$ is expressed by

$$c(T_{OG(l_{d+1},(\underline{W}_{d+1}^L)^\perp/\underline{W}_{d+1}^R)/Z^{(d)}})\cong c((\underline{U}_{d+1}/\underline{W}_{d+1}^L)^\vee\otimes(\underline{U}_{d+1}^\perp/\underline{U}_{d+1}) )c(\wedge^2(\underline{U}_{d+1}/\underline{W}_{d+1}^L)^\vee)$$
and the rest by $c(TZ^{(d)})=\prod_{i=1}^dc((\underline{U}_i/\underline{W}_i^L)^\vee\otimes(\underline{W}_i^R/\underline{U}_i)).$
Putting all together, we complete the proof.
\end{proof}

Since the constant $\gamma_{\underline{\alpha},\underline{\beta}}$ is obtained as in type D, we state Theorem \ref{lem5} without the proof.

\begin{theorem}\label{lem5}
Let $\underline{U}$ be the pullback of the universal tautological subbundle on $OG(n,V)$. Then the coefficient $\gamma_{\underline{\alpha}, \underline{\beta}}$ of the Schubert class $\left[\mathbb{S}(\underline{\beta})\right]$ in $\pi_*c_{SM}(Z_{\underline{\alpha}})$ is given by the integration
$$\gamma_{\underline{\alpha}, \underline{\beta}}=\int_{Z_{\underline{\alpha}}}c(TZ_{\underline{\alpha}})\cdot \widetilde{P}_{\rho(n)\backslash\underline{\beta}}(\underline{U}^\vee)\cap\left[Z_{\underline{\alpha}}\right].$$
\end{theorem}

We can carry out the computation for $\gamma_{\underline{\alpha},\underline{\beta}}$ as before either with Theorem \ref{CTZ_B} and Theorem \ref{lem5} or just by evaluating the weight for the basis $e_{\bar{n}}$ to be $0$.

The Chern-Mather class of Schubert varieties in the even orthogonal Grassmannian is closely related to the Kazhdan-Lusztig class in the odd orthogonal Grassmannian:
let $\mathbb{S}(\underline{\alpha})$ be a Schubert variety in $OG(n, \mathbb{C}^{2n+1})$ and $\mathbb{S}'(\underline{\alpha})$ a Schubert variety in $OG'(n+1,\mathbb{C}^{2n+2})$. The Kazhdan-Lusztig class $KL(\mathbb{S}(\underline{\alpha}))$ of $\mathbb{S}(\underline{\alpha})$ is equal to the Chern-Mather class $c_M(\mathbb{S}'(\underline{\alpha}))$ of the Schubert variety $\mathbb{S}'(\underline{\alpha})$, 
$$KL(\mathbb{S}(\underline{\alpha}))=c_M(\mathbb{S}'(\underline{\alpha}))$$
by the commutative diagram \eqref{comm} with Theorem \ref{thm:CM} and Theorem \ref{thm:KL}. 
Indeed, the Kazhdan-Lusztig class of $\mathbb{S}(\underline{\alpha})$ does agree with the Mather class of $\mathbb{S}(\underline{\alpha})$ in $OG(n,\mathbb{C}^{2n+1})$, because of the isomorphism of Schubert varieties between the types B and D Grassmannians.

%
%
%
%



\bibliographystyle{amsplain}

\end{document}